\documentclass[aop]{imsart2}
\usepackage{tikz}
\usepackage{neuralnetwork}
\usepackage{amsthm, amssymb, amsmath,verbatim,soul}
\usepackage{cite}
\usepackage{appendix}

\usepackage{mathptmx}
\usepackage[T1]{fontenc} 
\usepackage[utf8]{inputenc}
\usepackage{lmodern}

\theoremstyle{plain}

\newtheorem{theorem}{Theorem}[section]
\newtheorem{proposition}[theorem]{Proposition}
\newtheorem{corollary}{Corollary}[section]
\newtheorem{lemma}[theorem]{Lemma}
\theoremstyle{remark}

\def \a {\mathbf a}

\def \bP {\mathbb P}

\def \rme {\mathrm e}
\def \a {\alpha}

\def \E {\mathbb{E}}

\def \N {\mathbb{N}}

\def \P {\mathbb{P}}

\def \R {\mathbb{R}}

\def \cC {\mathcal{C}}

\def \cG {\mathcal{G}}

\def \cI {\mathcal{I}}

\def \cN {\mathcal{N}}

\def \cR {\mathcal{R}}

\def \cU {\mathcal{U}}
\def \cV {\mathcal{V}}

\def \lp {\left(}
\def \rp {\right)}
\def \ls {\left\{}
\def \rs {\right\}}

\def \rme{\mathrm{e}}

\DeclareMathOperator{\TW}{TW}
\DeclareMathOperator{\GUE}{GUE}

\def\osbd{{{\rm BD}^{\sharp}}}
\def\BLPP{D}
\def\PLPP{L}

{}

\usepackage{hyperref}

\begin{document}
\begin{frontmatter}

\title{GUE Fluctuations Near the Axis in One-Sided Ballistic Deposition}
\runtitle{Fluctuations in One-Sided Ballistic Deposition}

\begin{aug}
\author[A,B]{\fnms{Pablo} \snm{Groisman}\ead[label=e1, mark]{pgroisma@dm.uba.ar}},
\author[B]{\fnms{Alejandro F.} \snm{Ram\'irez}\ead[label=e2,mark]{ar23@nyu.edu}}
\author[C]{\fnms{Santiago} \snm{Saglietti}\ead[label=e3,mark]{saglietti.sj@uc.cl}}
\and
\author[A]{\fnms{Sebasti\'an} \snm{Zaninovich}\ead[label=e4,mark]{szaninovich@dm.uba.ar}}
\address[A]{Fac. Cs. Exactas y Naturales, Universidad de Buenos Aires and IMAS UBA-CONICET,
\printead{e1,e4}}

\address[B]{NYU-ECNU Institute of Mathematical Sciences at NYU Shanghai
\printead{e2}}

\address[C]{Facultad de Matemáticas, Pontificia Universidad Católica de Chile
\printead{e3}}

\end{aug}

\begin{abstract}
We introduce a variation of the classic ballistic deposition model in which vertically falling blocks can only stick to the top or the upper right corner of growing columns. We establish that fluctuations of the height function at points near the $t$-axis are given by the GUE ensemble and its corresponding Tracy-Widom limiting distribution. The proof is based on a graphical construction of the process in terms of a directed Last Passage Percolation model. Using this graphical construction, we define the notion of geodesics for the height function and show that the wandering exponent governing the transversal fluctuations of these geodesics is $2/3$.
\end{abstract}



\end{frontmatter}

\tableofcontents

\section{Introduction and main results}

Ballistic deposition is a mathematical model describing the growth of an interface due to the random accumulation of aggregates or particles in space. It has attracted extensive interest both in the physical sciences and mathematics community due to its simple description and its relevance as a physical model for random aggregation. Nevertheless, to this date, it has remained essentially mathematically intractable. 
Vold \cite{Vold} introduced the model as a tool to numerically calculate the sediment volume to be expected for a suspension that is diluted enough so that each particle is separated from every other and then allowed to settle quietly, under gravity. It has also been considered as a toy model for diffusion limited aggregation \cite{Atar}.

The process considered in the mathematical community can be informally described as follows: on each site $x \in \mathbb Z^d$, particles fall vertically at random forming columns which grow over time, in such a way that a particle falling on $x$ sticks to the top of the highest column among those growing either above $x$ or one of its neighboring sites.  
We call this model the $d+1$ dimensional case, for a precise definition see \cite{Penrose, Timo}. Here, we restrict ourselves to $d=1$, where the height of the columns describes a one-dimensional interface on the plane. {In what follows, we denote the height of this interface at time $t$ above site $x$ by $h(t,x)$.}

A law of large numbers and a scaling limit for the height of the interface are known to hold \cite{Timo, Penrose}, in particular proving that it grows linearly in time as $ct$ for some $c>0$.
With a law of large numbers at hand, it is natural to wonder about the fluctuations of the growing interface, which are believed to belong to the so-called KPZ universality class \cite{Cannizzaro, Remenik, Corwin, Quastel}. However, apart from heuristic arguments, almost no mathematical evidence of this  conjecture is available.




{Models in the KPZ universality class all have an analogue of our height function $h$ which is conjectured to have the following common behavior (see \cite{Remenik, Quastel, Corwin}): 
\begin{enumerate}
    \item Fluctuations on the height function at time $t$ are expected to be of order $t^{\chi}$, with $\chi=1/3$.
    \item Spatial correlations are expected to be of order $t^{\xi}$, where the exponents $\chi$ and $\xi$ obey the universal relation $\chi=2\xi -1$ in any dimension. For $\chi=1/3$, this gives $\xi=2/3$.
    \item For large times and under the so-called \textit{KPZ 1:2:3 scaling}, the height function is expected to converge to a limiting universal field $\mathfrak{h}$, i.e., as $r \to \infty$, 
\begin{equation}\label{eq:KPZ}
r^{-\frac{1}{3}}\Big( h(c_1rt, \lfloor c_2 r^{\frac{2}{3}}x\rfloor) - c_3 rt\Big) \overset{d}{\longrightarrow} \mathfrak{h}(t,x),
\end{equation}
where the constants $c_1$, $c_2$ and $c_3$ may depend on the particular model but the field $\mathfrak{h}$ does not, it only
depends on the initial condition. The universal field $\mathfrak{h}$ is known as the \textit{KPZ fixed point}, and its marginal distributions have been explicitly computed for some particular initial conditions, see \cite{Matetski}. In particular, for narrow wedge initial conditions, the marginals of the KPZ fixed point are given by the GUE Tracy-Widom distribution \cite{Remenik,Tracy}, so that the fluctuations of the height function rescaled at time~$t$ by $t^{1/3}$ are expected to converge to this law, which we denote hereafter by  $\TW_{\GUE}$.

\end{enumerate}

Despite the behavior of fluctuations remaining mostly conjectural so far (see the next paragraph), a handful of rigorous results have been established for the model, including: the linear growth of the interface \cite{Timo} and the convergence of the rescaled interface to the viscosity solution of a Hamilton-Jacobi equation \cite{Timo} obtained by Sepp\"al\"ainen; and the existence of an invariant distribution for the height process when centered around the height at the origin proved by Chaterjee \cite{Chaterjee2} (in fact, these results hold for every $d\ge 1$). Penrose and Yukich \cite{PenroseYukich} obtained central limit theorems and other results for the number of blocks deposited in a large region within a fixed time. A law of large numbers for a variant of the process, considered in a one-dimensional strip was proved by Atar in \cite{Atar}. Extensions and refinements of the results in \cite{Atar} have been obtained in \cite{Mansour, Braun}.

Concerning rigorous results on the fluctuations of ballistic deposition, a logarithmic lower bound for the variance of the height function has been proved by Penrose in \cite{Penrose} and
an upper bound of order $\sqrt{t/\log t}$ was obtained by Chatterjee in \cite{ChatterjeeSuper} for a variant of the model in which all heights are updated simultaneously at integer times and the block sizes are random. Random block sizes are also considered in \cite{Santi}, but in their setting the size of the blocks is heavy-tailed and, in addition, blocks stick to neighboring columns with probability $p$, which can be less than one.

Cannizzaro and Hairer considered another variant of the model in which the height at a specific location is updated to the height of a randomly selected neighbor. The law of the selected neighbor depends on a temperature parameter, with standard ballistic deposition corresponding to the zero-temperature case. This model is analyzed in great detail by the authors \cite{Cannizzaro}, proving in particular that the infinite temperature case belongs to a different (new) universality class.


In this work, we introduce the {\it one-sided ballistic deposition process}, which is a variant of $1+1$ standard ballistic deposition. In this one-sided version, particles stick only to one side of the blocks instead of both. As our main result, we establish that the height of the interface at points $(t,k)$ close to the $t$-axis has GUE fluctuations. To prove this, we use a toolbox developed to study directed percolation models \cite{Bodineau, Baik, Baryshnikov, Glynn, OConnell, OConnellYor, Chaterjee, Bates, Basu, Basu2}.

To state our results, let us give a precise
definition of the one-sided ballistic deposition process, to be denoted henceforth by $\osbd$. We start by considering a sequence $(Q^{(r)})_{r \in \N}$ of independent Poisson processes $Q^{(r)}=(Q^{(r)}_t)_{t \geq 0}$ of rate~$1$. Given $t \in \R_{\geq 0}$, we write $t \in Q^{(r)}$ to indicate that $t$ is a discontinuity point of $Q^{(r)}$. We also say that $t$ is a {\em mark} of $Q^{(r)}$. Now, for $G: \N \to [-\infty,\infty)$, the $\osbd$ process with initial condition $G$, denoted by $h_G=(h_G(t,k): t \geq 0\,,\,k \in \N)$, is defined recursively by setting $h_G(t,1):=G(1)+Q^{(1)}_t$ and then, for each $k \in \N$, defining $h_G(\cdot,k+1)$ as the right-continuous piecewise constant function which is discontinuous at a point $t$ if and only if $t \in Q^{(k+1)}$ and satisfies
\begin{equation}
\label{eq:definition.osbd}
h_G(t,k+1)=\begin{cases}G(k+1) & \text{ if $t=0$}\\ 1+ \max\{h_G(t^-,k),h_G(t^-,k+1)\} & \text{ if $t > 0$ and $t \in Q^{(k+1)}$,}\end{cases}
\end{equation}
see Figure \ref{fig:smallN}. A discrete-time version of this model was previously studied in \cite{Nagatani} by means of numerical simulations.
Two particular initial conditions $S,F$ play a special role for us. We call the configuration $S:=(-\infty)\mathbf{1}_{\N \setminus \{1\}}$  the {\it seed-type} or {\em narrow-wedge} initial condition (at $k=1$), while the configuration $F:=0$ is called the {\it flat} initial condition.

\begin{figure}
 \begin{center}
  \includegraphics[width=.45\textwidth]{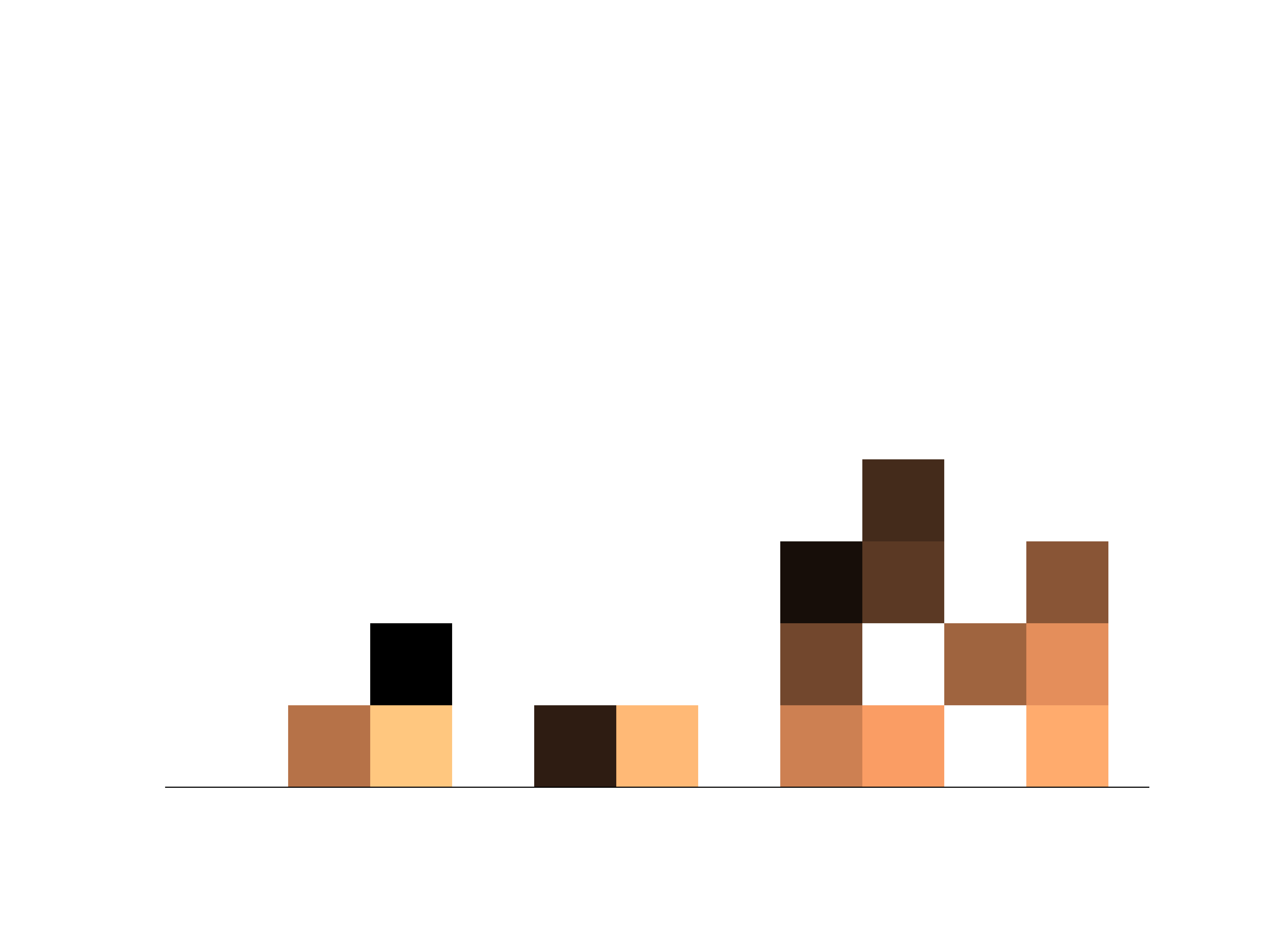}  \includegraphics[width=.5\textwidth]{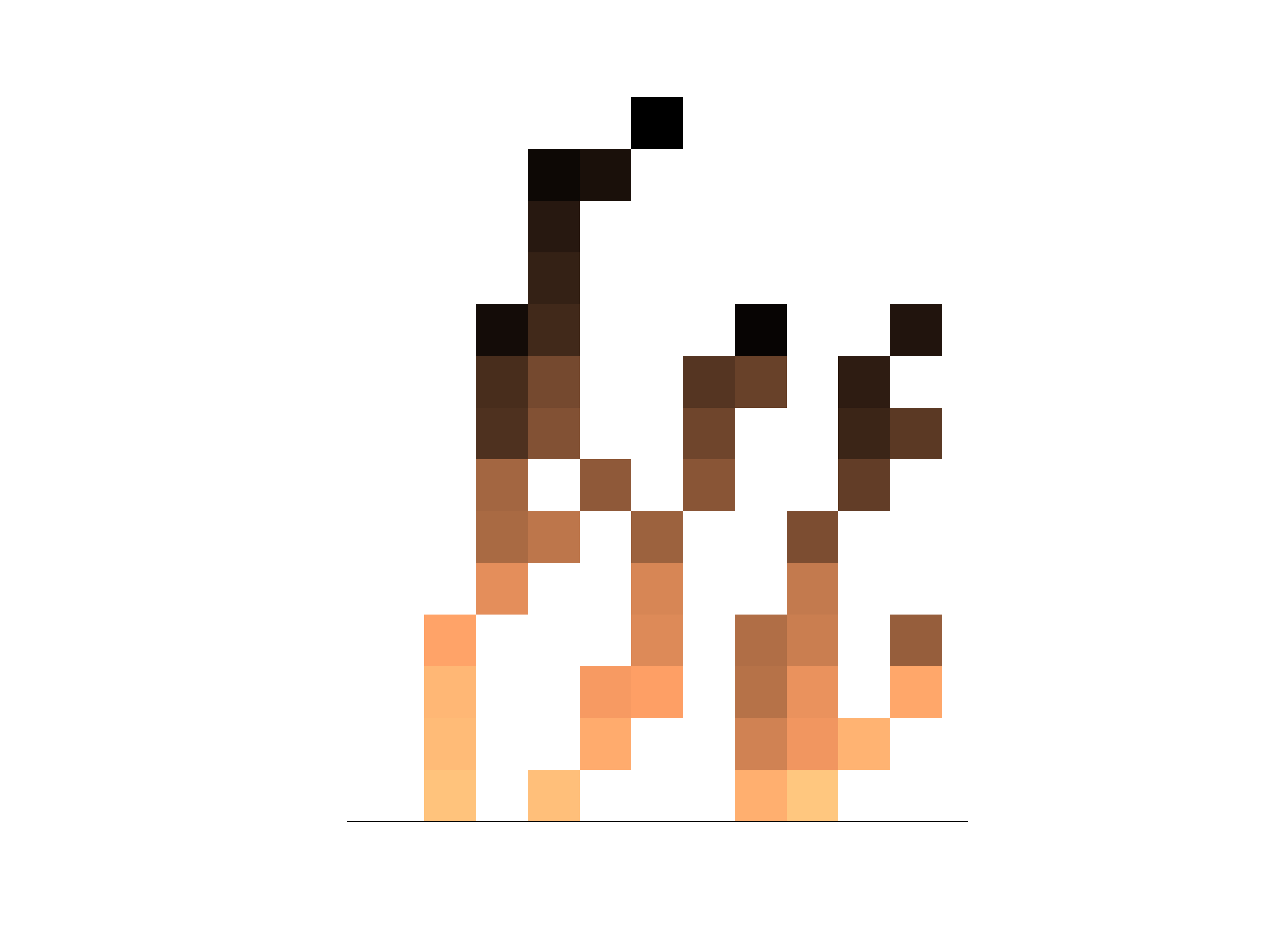}
\caption{Two realizations of the process with flat initial condition at time $t=1.5$ (left) and $t=5$ (right). The darker the color, the later the arrival of the block.}
\label{fig:smallN}
\end{center}
\end{figure}

We denote by $\cI$ the space of initial conditions $G$ such that $S \leq G \leq F$. In the sequel, $k$ always denotes a nonnegative integer (and we do not clarify this on every occasion). Also, $\lambda^{\max}_k$ stands for the largest eigenvalue in a $k \times k$ GUE random matrix and $\overset{d}{\longrightarrow}$ denotes convergence in distribution. {Finally, given a pair of functions $f,g:\R_{\geq 0} \to \R$, in the sequel we will write $f(t)=o(g(t))$ to mean that $\lim_{t \to \infty} \frac{f(t)}{g(t)}=0$.}

\begin{figure}
 \begin{center}
 \includegraphics[width=\textwidth]{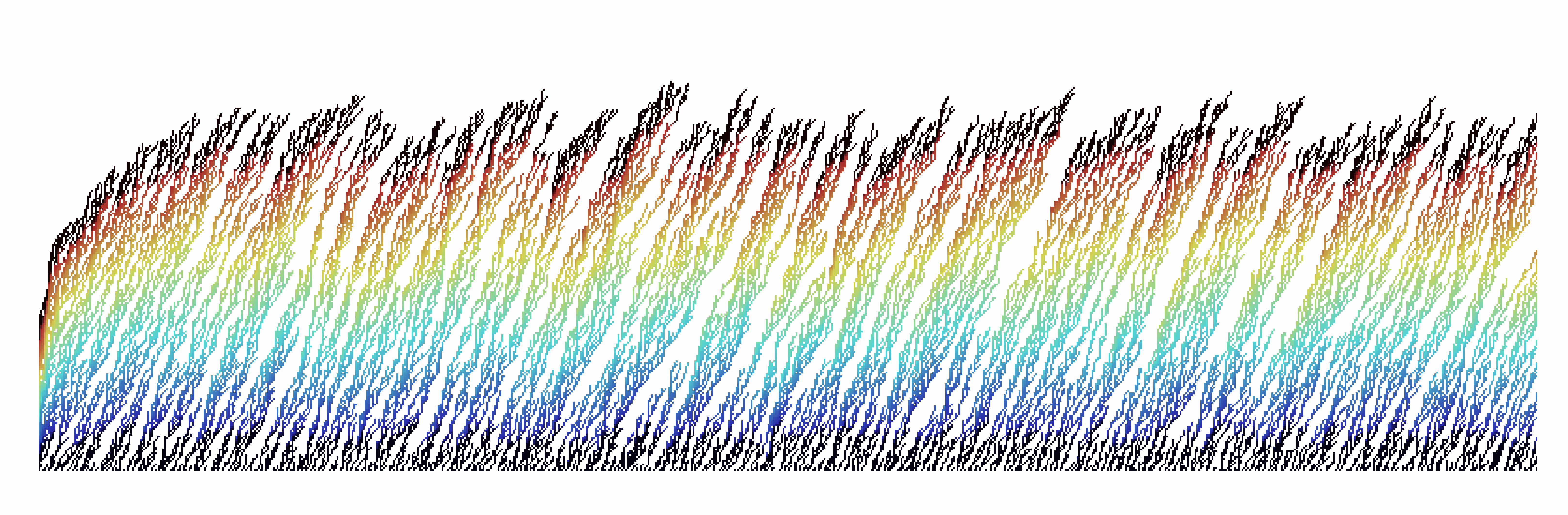}\\
\includegraphics[width=.5\textwidth]{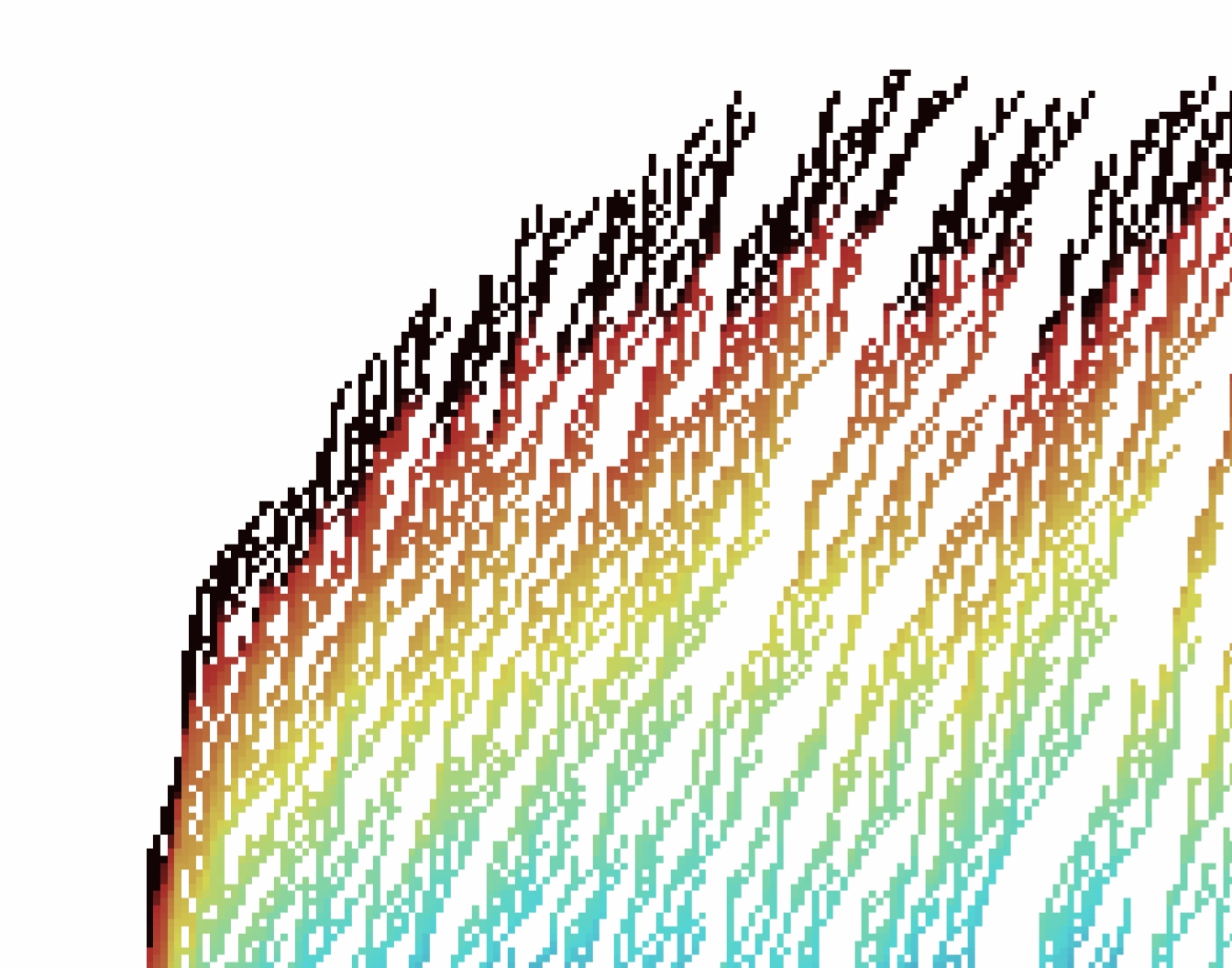}
\caption{A realization of the process with flat initial condition at time $t=75$ (top) and a zoom-in near the $t$-axis (bottom). On the bottom picture, the behavior $h_F(t,k) \sim  t + t^{1/2}\lambda_{k}^{\max}\sim t + t^{1/2}2\sqrt{k}$ for small values of $k$, predicted by Theorem~\ref{theo:1}, can be appreciated. See also Figure \ref{fig:seed} (right).}\label{fig:flat}
\end{center}
\end{figure}

\begin{figure}
 \begin{center}
 \includegraphics[width=.25\textwidth]{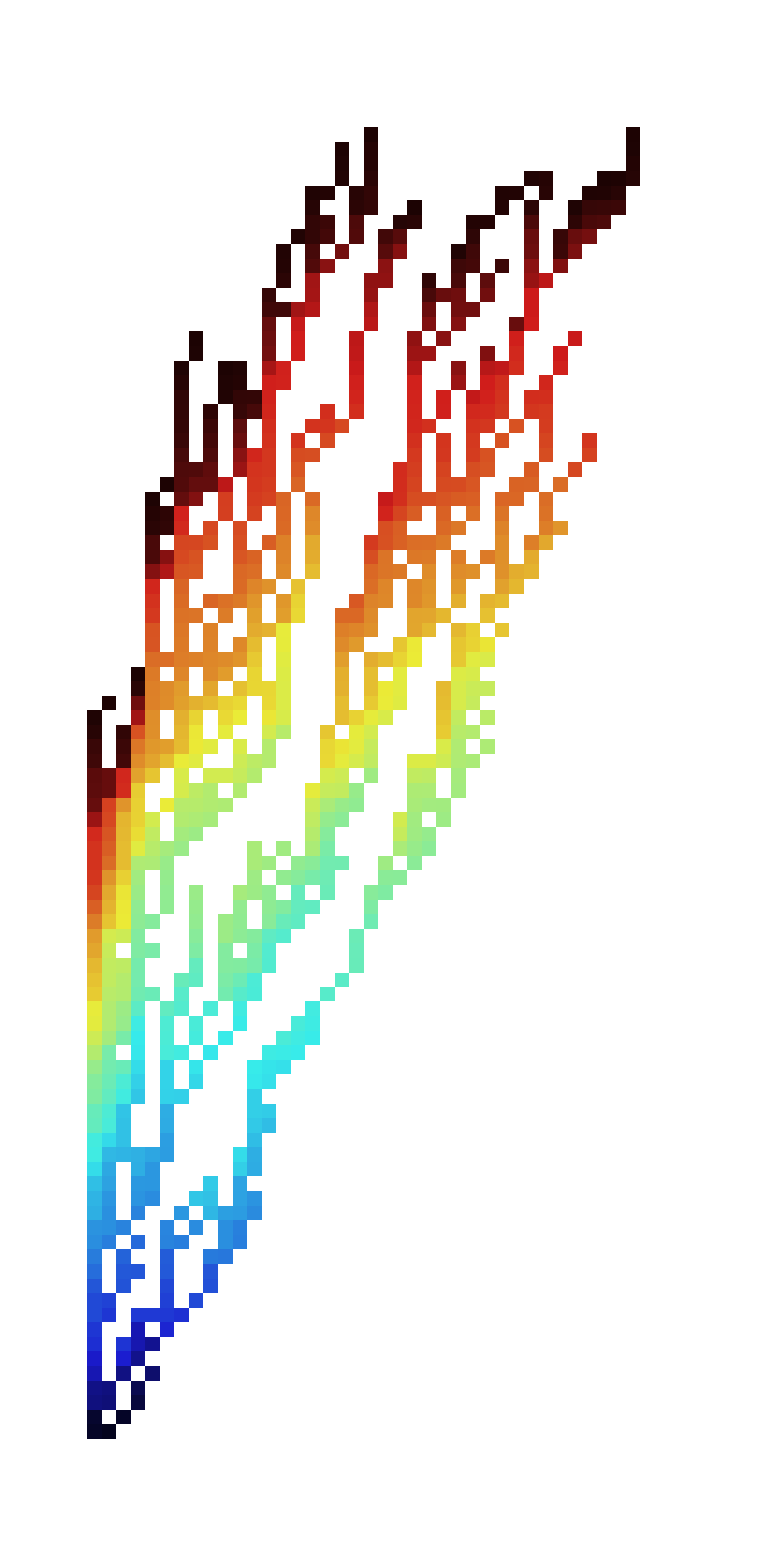}  \hspace{.2\textwidth}\includegraphics[width=.25\textwidth]{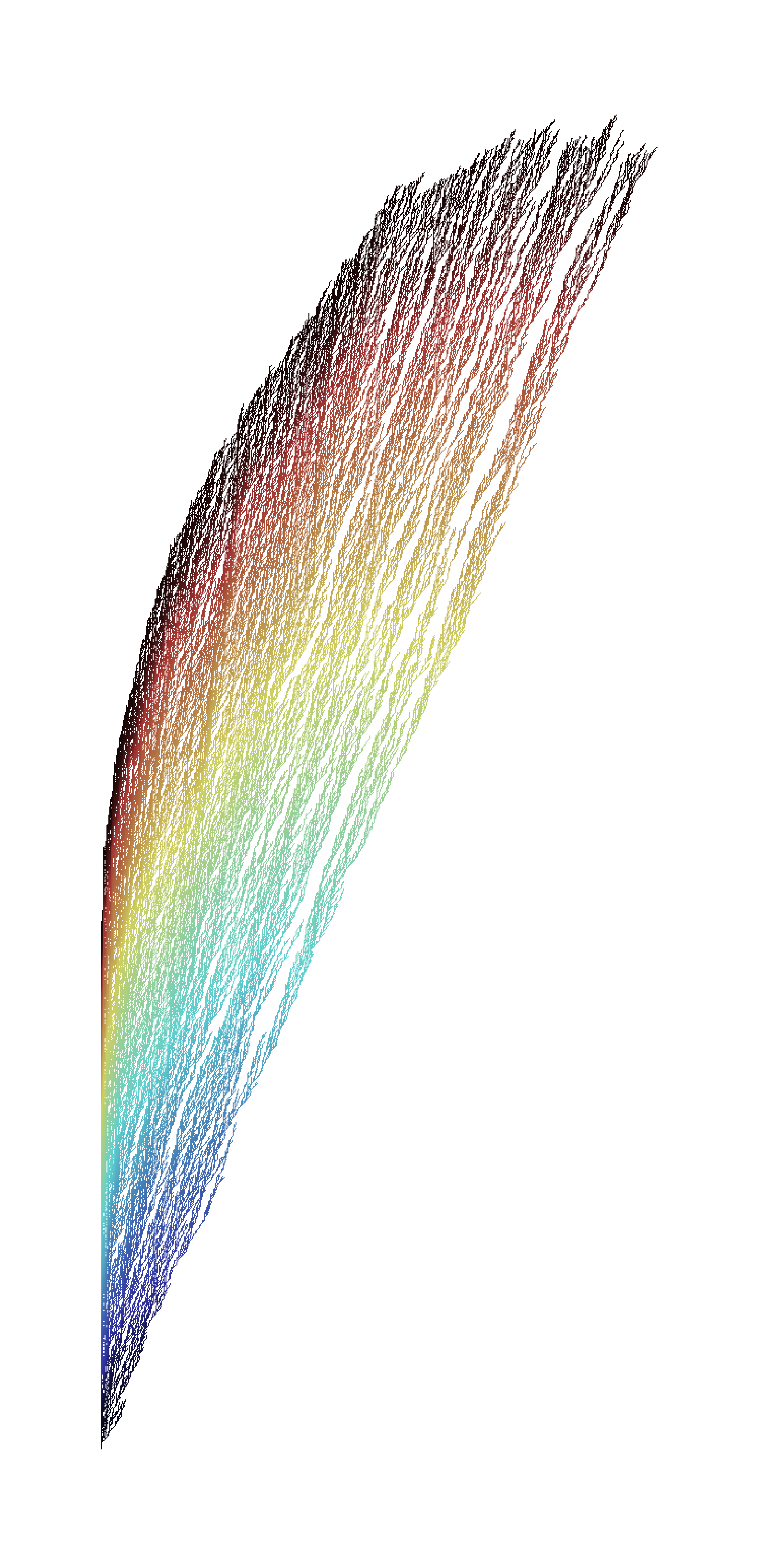}
\caption{Two realizations with seed-type initial condition at short times (left) and large times~(right). Different scales are used in each case.} \label{fig:seed} 
 \end{center}
\end{figure}

We now present the main results of this article. For clarity of the exposition, we split these results into two separate subsections: longitudinal and transversal fluctuations. 

\medskip 

\subsection{Longitudinal fluctuations} Our first main result establishes the leading order and asymptotic fluctuations of $h_G(t,k)$ for points close enough to the $t$-axis. We~direct the reader to Figures~\ref{fig:flat} and \ref{fig:seed} for an illustration.
\begin{theorem}\label{theo:1} Let $\alpha : \R_{\geq 0} \to \N$ be a given function and $h_G$ be a $\osbd$ process with initial condition $G \in \cI$. We have that:
\begin{enumerate}
    \item[i.] if $\lim_{t \to \infty} \alpha(t) = \infty$ and $\alpha(t)=o\big(t^{\frac{3}{7}}(\log t)^{-\frac{6}{7}}\big)$, then, as $t \to \infty$,  
\begin{equation}\label{eq:t1}
\frac{h_G(t,\alpha(t)) - t -2\sqrt{t \alpha(t)}}{\sqrt{t(\alpha(t))^{-\frac{1}{3}}}} \overset{d}{\longrightarrow} {\TW}_{\GUE}.
\end{equation}
\item[ii.] if $\lim_{t \to \infty} \alpha(t) = k \in \N$, then, as $t \to \infty$,
\[
\frac{h_G(t,\alpha(t)) - t }{\sqrt{t}} \overset{d}{\longrightarrow} \lambda^{\max}_k.
\]
\item [iii.] In general, under the sole condition that $\alpha(t) \log t \leq t$ for all $t$ large enough, we have 
\begin{equation}\label{eq:expansion}
h_G(t,\alpha(t))=t+2\sqrt{t\alpha(t)}+R_G(t),
\end{equation} where the error term $R_G(t)$ satisfies $R_G(t)=O\left(\max\left\{\alpha(t)\log t, \sqrt{t(\alpha(t))^{-\frac{1}{3}}}\right\}\right)$, i.e., for any $\beta: \R_{\geq 0} \to \R_{\geq 0}$ such that $\max\left\{\alpha(t)\log t, \sqrt{t(\alpha(t))^{-1/3}}\right\}=o(\beta(t))$, 
\[
\frac{\sup_{G \in \cI} |R_G(t)|}{\beta(t)} \overset{\P}{\longrightarrow} 0.
\] In particular, if $\lim_{t\to\infty}\alpha(t)=\infty$ and $\alpha(t)= o(\sqrt{t}(\log t)^{-1})$, then, as $t\to\infty$,
\begin{equation}\label{eq:subg}
\sup_{G \in \cI} \left|\frac{h_G(t,\alpha(t))-t-2\sqrt{t\alpha(t)}}{\sqrt{t}}\right|\overset{\P} {\longrightarrow} 0.
\end{equation}
\end{enumerate}
\end{theorem}

The limit in (i) of Theorem \ref{theo:1} constitutes the first instance in which a model close to ballistic deposition is shown to have features associated with the KPZ universality class at this level.
The unusual denominator $\sqrt{t(\alpha(t))^{-\frac{1}{3}}}$ is a consequence of the fact that we are considering distances in the regime $\alpha(t) = o(t^{3/7-})$ instead of $\alpha(t)=xt$ with $x>0$. If we could take $\alpha(t)=xt$ in \eqref{eq:t1} (which Theorem~\ref{theo:1} does not allow us to),  we would then recover the usual linear leading order term in \eqref{eq:KPZ} and the conjectured exponent $\chi=1/3$ for the longitudinal fluctuations. On the other hand, the limit in~\eqref{eq:subg} discards a diffusive scaling for the height function, compatible with Gaussian behavior, for $\alpha(t)=o(\sqrt t(\log t)^{-1})$. Up to our knowledge, this is yet to be confirmed for standard ballistic deposition.
Furthermore, (iii) would indicate, if one could show the error term in \eqref{eq:expansion} satisfies $R_G(t)=o(\alpha(t))$, that the following scaling limit holds:
\begin{equation}\label{eq:lln}
\lim_{n\to\infty}\frac{h_G(nt,nx)}{n}=t+2\sqrt{tx}.
\end{equation}
Nonetheless, we do not believe that proving Theorem~\ref{theo:1} is necessary to obtain \eqref{eq:lln}. 
Note the presence of a parabola in the spatial variable $x$ given by the second term $2\sqrt{tx}$ above. This is consistent with the simulations of Figures \ref{fig:flat} and \ref{fig:seed}, where a parabola can be observed for spatial distances $x$ of the same or smaller order than time $t$, i.e., $x\le t$.


In the context of last passage percolation (LPP), Bodineau and Martin~\cite{Bodineau}
and Baik and Suidan \cite{Baik} obtained independently a result similar to Theorem~\ref{theo:1} (i). The proofs of these results rely on the well-known fact that the distribution of the largest eigenvalue of a GUE matrix agrees with that of the last passage time in Brownian LPP \cite{Baryshnikov}, which in turn is close to standard LPP in the regime they study. 
More recently, the model of Brownian LPP has been studied in much more detail in \cite{Dauvergne, Basu, Basu2} (among other works), where, in particular, sharp moderate deviation estimates together with a law of fractional logarithm were obtained \cite{Basu, Basu2} and a full scaling limit of Brownian LPP towards an object known as the directed landscape was proved \cite{Dauvergne}. 
In this regard, we have the following moderate deviations result, which can be seen as the counterpart for $\osbd$ of those in \cite{Basu2,PaquetteZeitouni}. None of these estimates are necessary to obtain Theorem~\ref{theo:1}, but they will be used in the study of transversal fluctuations of $\osbd$.

\begin{theorem}\label{theo:3} Let $\alpha: \R_{\geq 0} \to \N$ satisfy $\lim_{t \to \infty} \alpha(t) = \infty$ and $\alpha(t)=o\big(t^{\frac{3}{7}}(\log t)^{-\frac{6}{7}}\big)$. Then, given $\varepsilon > 0$, there exist constants $t_\varepsilon,k_\varepsilon,\gamma_\varepsilon, x_\varepsilon > 0$ such that, for any $G \in \cI$, $t > t_\varepsilon$ and $k \in [k_\varepsilon,\alpha(t)]$, if $x \in [x_\varepsilon,\min \{ \gamma_\varepsilon k^{2/3},  (\log t)^2\}]$ we have 
\begin{equation}\label{eq:est1}
\exp\left(-\frac{4}{3}(1+\varepsilon)x^{3/2}\right)\leq \bP\left( \frac{h_G(t,k)-t - 2\sqrt{tk}}{\sqrt{tk^{-\frac{1}{3}}}} \geq x\right) \leq \exp\left( -\frac{4}{3}(1-\varepsilon)x^{3/2}\right),
\end{equation} and if $x \in [x_\varepsilon, \min \{k^{1/10},\sqrt{\log t}\}]$ we have
\begin{equation}\label{eq:est2}
\exp\left(-\frac{1}{12}(1+\varepsilon)x^3\right)\leq \bP\left( \frac{h_G(t,k)-t - 2\sqrt{tk}}{\sqrt{tk^{-\frac{1}{3}}}} \leq -x\right) \leq \exp\left( -\frac{1}{12}(1-\varepsilon)x^3\right).
\end{equation}
Furthermore, if $\alpha(t)=o\big(t^{\frac{3}{7}}(\log t)^{-\frac{24}{7}}\big)$, there exist constants $C,t^*,\gamma^*,x^* > 0$ such that, for any $G \in \cI$, $t > t^*$, $k \in [1,\alpha(t)]$ and $x \in [x^*,\min\{ \gamma^* k^{1/6}, (\log t)^2\}]$, we have, in fact, the following sharper estimate for the right tail:
\begin{equation}
\label{eq:est3}
\frac{1}{C} x^{-3/2} \rme^{-\frac{4}{3}x^{3/2}} \leq \bP\left( \frac{h_G(t,k)-t - 2\sqrt{tk}}{\sqrt{tk^{-\frac{1}{3}}}} \geq x\right) \leq C x^{-3/2} \rme^{-\frac{4}{3}x^{3/2}}.
\end{equation}
\end{theorem} 

To prove Theorem~\ref{theo:1} and Theorem~\ref{theo:3}, we exploit an alternative representation of $\osbd$ in terms of an LPP-like process,
{which reduces} our problem to that of studying the fluctuations in a directed LPP processes similar (but different) to the one considered in \cite{Bodineau, Baik}. To continue stating our other main results, it will be convenient to introduce this alternative representation of $\osbd$ now.

Given a sequence $Y=(Y^{(r)})_{r \in \N}$ of independent Poisson processes $Y^{(r)}=(Y^{(r)}_t)_{t \geq 0}$ of rate $1$, we write $u \in Y^{(r)}$ to indicate that $u \in \R_{\geq 0}$ is a discontinuity point of $Y^{(r)}$, i.e., $Y^{(r)}_u \neq Y^{(r)}_{u^-}$, and, for $t > 0$ and $k \in \N$, let us write $\cU(t,k)$ to denote the space of all c\` adl\`ag increasing paths starting at $1$ that have at most $k$ jumps on $[0,t]$, all of them with size $+1$, which can only jump from a value $r$ to $r+1$ at a given time $s$ if $s$ is a mark of $Y^{(r)}$ (see \eqref{eq:defu} for a precise definition). 
Then, for any initial condition $G:\N \to [-\infty,\infty)$ define the LPP-like process $h^\circ_G:=(h^\circ_G(t,k) : t \geq 0\,,\,k \in \N)$ by the formula
\begin{equation}\label{eq:defhcirc}
h^\circ_G(t,k):= \sup_{u \in \cU(t,k)} [H(u,Y)+G(k-u(t)+1)],
\end{equation} where $H(u,Y)$ denotes the total amount of marks belonging to $Y=(Y^{(r)})_{r \in \N}$ which are contained in the (closure of) the graph of $u$ on $[0,t]$ (see \eqref{eq:defh} for a precise definition).

We will show in {Lemma~\ref{lemma:1}} below that, for any fixed $t > 0$ and $k \in \N$, we have that
\begin{equation}\label{eq:defhcirc2}
h_G(t,k) \overset{d}{=} h_G^\circ(t,k).
\end{equation}
Analogous representations have been used previously in \cite{Santi, Khanin, Cannizzaro, Sudijono} in various contexts. In particular, it is not exclusive to $\osbd$ and holds also for standard (two-sided) ballistic deposition, see \cite{Santi}. However, so far we have not been able to produce an effective comparison between standard ballistic deposition and a Brownian LPP-like process, which is why we consider here one-sided ballistic deposition instead.

It is this equality in distribution that will allow us to obtain Theorem~\ref{theo:1} by studying instead the process $h_G^\circ$ using LPP techniques. However, it is worth pointing out that, while their marginal distributions coincide, $h_G$ and $h_G^\circ$ do not have the same distribution as processes in general. Indeed, this will be evident from the proof of Lemma~\ref{lemma:1} but, in addition, can be seen from the fact that $h_F^\circ(t,k) \leq h_F^\circ(t,k+1)$ for all $t > 0$ and $k \in \N$ by mere definition, whereas this is not true for $h_F$: the inequality $h_F(t,k) \leq h_F(t,k+1)$ does hold but only in a distributional sense, not almost surely, as evidenced in Figure~\ref{fig:flat}.

\subsection{Transversal fluctuations} \label{sec:transversal} Our final result characterizes the order of fluctuations of optimizing paths for $h_G^\circ$. Before we can state it, we need some definitions.

For $\gamma,t > 0$ and $k \in \N$, we define the cylinder 
\[
    C^\gamma(t,k) := \left\{ (s,n) \in [0, t] \times \N : \left| n - \frac{k}{t}s\right| \leq k^\gamma \right\},
\] together with, for $G \in \cI$, the set $\Pi_G(t,k)$ of \textit{geodesics} for $h_G^\circ(t,k)$ as
\[
    \Pi_G(t,k) := \left\{ u \in \cU(t,k) :  H(u,Y) + G(k-u(t)+1) = h^\circ_G(t,k)\right\},
\] where $\cU(t,k)$ and $H(u,Y)$ are as in \eqref{eq:defhcirc}. Next, let us define the event $A_G^\gamma(t,k)$ in which all geodesics for $h^\circ_G(t,k)$ are contained in $C^\gamma(t,k)$, i.e.
\begin{equation}\label{eq:defa}
A_G^\gamma(t,k) := \{ \text{graph}(u) \in C^\gamma(t,k) \text{ for all }u \in \Pi_G(t,k)\},
\end{equation} where, for any $u \in \cU(t,k)$, we write $\text{graph}(u):=\{(s,u(s)) : s \in [0,t]\}$ to denote the graph of $u$. Similarly, for each $s \in (0,1]$ we define the event
\begin{equation}\label{eq:defb}
 B_G^\gamma(t,k;s):=\{ \exists \,u \in \Pi_G(t,k) : |u(st) - ks| \leq k^\gamma \}
\end{equation} in which at least one of the geodesics for $h^\circ_G(t,k)$ has at time $t'=st$ a local deviation from the straight line $x=\frac{k}{t}t'$  which is less or equal than $k^\gamma$. 
Finally, for any $\alpha : \R_{\geq 0} \to \N$, we define the \textit{global (upper) transversal fluctuations exponent} as
\[
\xi_G(\alpha):=\inf \{ \gamma > 0 : \liminf_{t \to \infty} \bP(A_G^\gamma(t,\alpha(t)))=1\}.
\] and, for $s \in (0,1]$, the \textit{local (lower) transversal fluctuations exponent} as
\[
\xi_G(\alpha;s):= \inf \{ \gamma > 0 : \limsup_{t \to \infty} \bP(B_G^\gamma(t,\alpha(t);s))=1\}.
\] Our third and final main result is then the following.

{\begin{theorem}\label{theo:4} Let $\alpha: \R_{\geq 0} \to \N$ satisfy $\lim_{t \to \infty} \alpha(t)=\infty$ and $\alpha(t)=o\big(t^{3/7}(\log t)^{-6/7}\big)$. Then, for any $G \in \cI$ and $s \in (0,1)$, we have
\[
\frac{2}{3}\leq \xi_G(\alpha;s) \leq \xi_G(\alpha).
\] Furthermore, if in addition $\alpha$ satisfies $(\log(t))^\rho=o(\alpha(t))$ for all $\rho>0$ and $\alpha(t)=o\big(t^\eta\big)$ for some $\eta < \frac{9}{31}$, then we also have $\xi_G(\alpha) \leq \frac{2}{3}$, so that 
\[
\xi_G(\alpha;s)=\frac{2}{3} = \xi_G(\alpha).
\]
\end{theorem}}

Notice that, in particular, Theorem~\ref{theo:4} implies, whenever $\alpha(t)=\lfloor t^\eta\rfloor $ with $\eta \in (0,\frac{9}{31})$, that:
\begin{itemize}
    \item [$\bullet$] if $\gamma > \frac{2}{3}$, then, with probability tending to $1$ as $t \to \infty$, all geodesics for $h^\circ_G(t,\alpha(t))$ are contained in the cylinder $C^\gamma(t,\alpha(t))$;
    \item [$\bullet$] if $\gamma < \frac{2}{3}$, then, for any $s \in (0,1)$, with probability bounded away from $0$ as $t \to \infty$, all geodesics for $h^\circ_G(t,\alpha(t))$ lie outside $C^\gamma(t,\alpha(t))$ at time $t'=st$ (so that, in particular, no geodesic for $h^\circ_G(t,\alpha(t))$ is contained in $C^\gamma(t,\alpha(t))$). 
\end{itemize}

{The reason why we require the stronger condition $\alpha(t)=o(t^\eta)$ to show that $\xi_G(\alpha)\leq \frac{2}{3}$ is that, if we only assume that $\alpha(t)=o\big(t^{3/7}(\log t)^{-6/7}\big)$, then with our current methods we cannot rule out the possibility that geodesics escape the cylinder $C^\gamma(t,\alpha(t))$ near either of their two endpoints. This is because Theorem~\ref{theo:3} holds only for curves $\alpha$ which grow much slower than linearly in $t$. If one could improve the range of $\alpha$ for which Theorem~\ref{theo:2} holds, then this would allow us to take larger values of $\eta$ in Theorem~\ref{theo:4}.}

We point out that, even if the notion of geodesic has been defined (so far) for $h^\circ_G$ and not the original process $h_G$, for each $t > 0$ and $k \in \N$ there exists a notion of geodesics for $h_G(t,k)$ which are in direct correspondence with those in the set $\Pi_G(t,k)$ (see \eqref{eq:rep1}). In particular, the analogue of Theorem~\ref{theo:4} for geodesics in $h_G$ holds as well. However, geodesics for $h_G$ are not as straightforward to define as for $h_G^\circ$ and, thus, for this reason we opted to state Theorem~\ref{theo:4} for $h_G^\circ$ rather than $h_G$.

\medskip

{\noindent \textbf{Organization of the paper}. The paper is organized as follows. Section~\ref{sec:outline} contains the outlines of the proofs of Theorems~\ref{theo:1}--\ref{theo:3}, and Section~\ref{sec:outline2} that of Theorem~\ref{theo:4}. Then, Section~\ref{preliminaries.long.fluct} contains all preliminary results needed to prove Theorems~\ref{theo:1} and \ref{theo:3}, while Section~\ref{sec:prel2} contains those for Theorem~\ref{theo:4}. Finally, Section~\ref{sec:proofoftheo1} contains the proof of Theorem~\ref{theo:1},  Section~\ref{sec:proofoftheo3} presents the proof of Theorem~\ref{theo:3} building on previous~work and Section~\ref{sec:proofoftheo4} gives the proof of Theorem~\ref{theo:4}.} The reader interested solely in the proof of Theorem \ref{theo:1} may concentrate on Sections \ref{sec:outline}, \ref{preliminaries.long.fluct} and \ref{sec:proofoftheo1} and skip the other sections since they are not necessary to complete its proof.

\section{Outline of the proofs of Theorem~\ref{theo:1} and Theorem~\ref{theo:3}}\label{sec:outline}

Our strategy to prove Theorems~\ref{theo:1}--\ref{theo:3} is to compare the process $h_G$ with Brownian LPP for which analogous asymptotics have already been derived in \cite{Glynn,Baryshnikov, Basu}. The strategy of comparing a model to Brownian LPP to obtain GUE-type fluctuations, or even other statistics of the process, was previously implemented 
in \cite{Bodineau, Baik, Suidan} in the context of standard LPP. Our method of proof bears some similarities with these works, although additional work is required in our setting due to the differences between our model and standard LPP.

To formally define Brownian LPP, consider a sequence $B=(B^{(r)})_{r \in \N}$ of independent standard Brownian motions $B^{(r)}=(B^{(r)}_t)_{t \geq 0}$ and, for $t > 0$ and $k \in \N$, define the space of \textit{directed paths on $[0,t]$ ending at most at $k$} as
\begin{equation}\label{eq:defvtk}
\cV(t,k):=\{ v : [0,t] \to \N \,|\, v \text{ c\`adl\`ag increasing}\,,\,v(t)\leq k\}.
\end{equation}
The {\em Brownian LPP} model is then defined as $\BLPP=(\BLPP(t,k) : t > 0\,,\,k \in \N)$, where
\begin{equation}\label{eq:defl}
\BLPP(t,k):= \sup_{v \in \cV(t,k)} H(v,B),
\end{equation} and, for any sequence $\mathcal{F}=(f^{(r)})_{r \in \N}$ of c\`adl\`ag functions $f^{(r)}=(f^{(r)}_t)_{t \geq 0}$ and $v \in \cV(t,k)$, we define  
\begin{equation}\label{eq:defh}
H(v,\mathcal{F})
:=\int_0^t \mathrm{d}f^{(v(s))}_s := \sum_{r=1}^k f^{(r)}_{v_r}-f^{(r)}_{v_{r-1}},
\end{equation} where, for $r=0,\dots,k$, we write $v_r:=\inf \{ s \in [0,t] : v(s) > r\} \wedge t$, with the convention that $\inf \emptyset:=\infty$ (used whenever $v(t)<k$ and also so that we always have that $v_k=t$). We point out that, by the continuity of Brownian paths, maximizing over all of $\cV(t,k)$ coincides with maximizing only over paths in $\cV(t,k)$ such that $v(0)=1$ and $v(t)=k$. While the latter option yields the usual definition of Brownian LPP, see e.g. \cite{Dauvergne}, our definition is more convenient in the proofs to follow.
Furthermore, notice that any $v \in \cV(t,k)$ can be (essentially) reconstructed from the vector $(v_0,\dots,v_k)$. Indeed, given $s \in [0,t)$, we have $v(s)=\inf\{ r : v_r > s\} \wedge k$. The precise value of $v(t)$ cannot be recovered as it is not possible to determine by looking at $(v_0,\dots,v_k)$ whether $v$ has jumped at time $t$ or not, but this does not affect the value of $H(v,\mathcal{F})$ and is therefore irrelevant. For this reason,
in the sequel we will view elements of $\cV(t,k)$ either as paths as in \eqref{eq:defvtk} or as vectors $(v_0,\dots,v_k) \in \R^{k+1}$ such that $0=v_0 \leq v_1 \leq \dots \leq v_k=t$, choosing in each occasion whichever option is more convenient.

It is well known (see \cite{Baryshnikov}) that, for any $k \in \N$, $\BLPP(1,k)$ is distributed as $\lambda^{\max}_k$, the largest eigenvalue of a $k \times k$ GUE random matrix. Moreover, by Brownian scaling $\BLPP(t,k)$, has the same law as $\sqrt{t}\BLPP(1,k)$ for any $t > 0$. Combining these facts with the asymptotics as $k \to \infty$ for $\lambda^{\max}_k$ (see, e.g., \cite{TracyCMP, Forrester}),
\[
k^{\frac{1}{6}}(\lambda^{\max}_k - 2\sqrt{k}) \overset{d}{\longrightarrow} \TW_{\GUE},
\] one obtains the following result, whose first item can already be found in \cite{Bodineau, Baik}.

\begin{theorem}[\cite{TracyCMP}]\label{theo:DBP} For $\alpha: \R_{\geq 0} \to \N$, we have:
\begin{enumerate}
    \item [A.] if $\lim_{t \to \infty} \alpha(t)=\infty$, then, as $t \to \infty$,
\[
\frac{\BLPP(t,\alpha(t))-2\sqrt{t\alpha(t)}}{\sqrt{t(\alpha(t))^{-\frac{1}{3}}}} \overset{d}{\longrightarrow} \TW_{\GUE};
\]   
\item [B.] if $\lim_{t \to \infty} \alpha(t)=k \in \N$, then, as $t \to \infty$,
    \[
    \frac{\BLPP(t,\alpha(t))}{\sqrt{t}} \overset{d}{\longrightarrow} \lambda^{\max}_k.
    \]
\end{enumerate}
\end{theorem} 

In light of Theorem~\ref{theo:DBP}, we can obtain Theorem~\ref{theo:1} and Theorem~\ref{theo:3} from the former once the following comparison result is established.

\begin{theorem}\label{theo:2} There exist constants $c_1,c_2,t^* > 0$ such that, given any $t>t^*$ and $k \leq \frac{t}{\log t}$, there exists a coupling of $(h_G(t,k) : G \in \cI)$ and $\BLPP(t,k)$ such that, for any~$x > 0$,
\begin{equation}\label{eq:comp}
\bP\left( \sup_{G \in \cI} |h_G(t,k) - t - \BLPP(t,k)| > x \right) \leq \rme^{c_1 \log t - c_2\frac{x}{k}} + \rme^{-\frac{1}{2}k \log t}.
\end{equation} 
\end{theorem} 

In particular, as a consequence of the above result we obtain the following corollary, whose proof is straightforward and is thus omitted.

\begin{corollary}\label{cor:0} Given $\alpha: \R_{\geq 0} \to \N$ such that $\alpha(t)\log t \leq t$ for all $t$ large enough, under the coupling of Theorem~\ref{theo:2} we have that, as $t \to \infty$, 
\begin{equation}\label{eq:cotacor}
\frac{\sup_{G \in \cI} |h_G(t,\alpha(t)) - t - \BLPP(t,\alpha(t))|}{\beta(t)} \overset{\bP}{\longrightarrow} 0
\end{equation} for any $\beta: \R_{\geq 0} \to \R_{\geq 0}$ such that $\alpha(t)\log t = o(\beta(t))$.
\end{corollary}

In particular, it follows from Corollary~\ref{cor:0} that:
\begin{itemize}
\item [a)] if $\lim_{t \to \infty} \alpha(t) = \infty$ and $\alpha(t)=o\big(t^{\frac{3}{7}}(\log t)^{-\frac{6}{7}}\big)$, then $\alpha(t)\log t =o\big(\sqrt{t}(\alpha(t))^{-\frac{1}{6}}\big)$ and therefore, as $t \to \infty$,
\[
\frac{h_G(t,\alpha(t)) - t - \BLPP(t,\alpha(t))}{\sqrt{t(\alpha(t))^{-\frac{1}{3}}}} \overset{\bP}{\longrightarrow} 0;
\]
\item [b)] if $\lim_{t \to \infty} \alpha(t) = k \in \N$, then $\alpha(t) \log t =o(\sqrt{t})$ and therefore, as $t \to \infty$,
\[
\frac{h_G(t,\alpha(t)) - t - \BLPP(t,\alpha(t))}{\sqrt{t}} \overset{\bP}{\longrightarrow} 0.
\]
\end{itemize} 
Combining (a) with (A) and (b) with (B)
gives (i) and (ii) respectively in Theorem~\ref{theo:1} at once, while (iii) in Theorem~\ref{theo:1} follows from combining (A) with \eqref{eq:cotacor}. {Similarly, Theorem~\ref{theo:3} follows from combining \eqref{eq:comp} with the corresponding moderate deviation estimates for Brownian LPP found in \cite{Basu,PaquetteZeitouni}.}

We prove Theorem~\ref{theo:2} by comparing our $\osbd$ process with Brownian LPP in three subsequent steps. First, we show that all $\osbd$ processes with initial conditions in $\cI$  are close enough to each other in the appropriate scale.

\begin{proposition}\label{prop:1} There exist constants $c_1,c_2,t^* > 0$ such that, given any $t>t^*$ and $k \leq \frac{t}{\log t}$, there exists a coupling of the random variables $(h_G(t,k) : G \in \cI)$ such~that, for all $x > 0$,
\begin{equation}\label{eq:c1}
\bP\left( \sup_{G \in \cI} |h_G(t,k)-h_F(t,k)| > x \right) \leq \rme^{c_1 k \log t - c_2 x} + \rme^{-\frac{1}{2}k \log t}.
\end{equation}
\end{proposition} 

Then, we show that $h_F$, the deposition process with flat initial condition, is close enough to an auxiliary  LPP-like process
which we introduce next.

Given a sequence $Y=(Y^{(r)})_{r \in \N}$ of independent Poisson processes $Y^{(r)}=(Y^{(r)}_t)_{t \geq 0}$ of rate $1$, we define the \textit{auxiliary process} $\PLPP=(\PLPP(t,k) : t > 0\,,\,k \in \N)$ by
\begin{equation}\label{eq:defd}
\PLPP(t,k):= \sup_{v \in \cV(t,k)} H(v,Y),
\end{equation} with $\cV(t,k)$ given by \eqref{eq:defvtk} and $H$ as in \eqref{eq:defh}, i.e., 
\begin{equation}
\label{eq:defhy}
H(v,Y):=\int_0^t \mathrm{d}Y^{(v(s))}_s:=\sum_{r=1}^k Y^{(r)}_{v_r} - Y^{ (r)}_{v_{r-1}}.
\end{equation} The comparison result between $h_F$ and $\PLPP$ that we show is the following.

\begin{proposition}\label{prop:2} There exist constants $c_1,c_2,t^* > 0$ such that, given any $t>t^*$ and $k \leq t$, there exists a coupling of $h_F(t,k)$ and $\PLPP(t,k)$ such that, for all $x > 0$,
\begin{equation}\label{eq:c2}
\bP( |h_F(t,k) - \PLPP(t,k)| > x) \leq \rme^{c_1 \log t - c_2 \frac{x}{k}}.
\end{equation}
\end{proposition}

As a last step, we show that the fluctuations of $\PLPP$ (centered) and $\BLPP$ are close enough to each other.

\begin{proposition}\label{prop:3} There exist constants $c_1,c_2,t^* > 0$ such that, given any $t>t^*$ and $k \in \N$, there exists a coupling of $\PLPP(t,k)$ and $\BLPP(t,k)$ such that, for all $x > 0$,
\begin{equation}\label{eq:c3}
\bP( 
|\PLPP(t,k) - t - \BLPP(t,k)| > x) \leq \rme^{ c_1 k \log t - c_2 x}.
\end{equation}
\end{proposition}

It will be clear from the proofs of these propositions that it is possible to jointly couple the random variables $(h_G(t,k) : G \in \cI)$, $\PLPP(t,k)$ and $\BLPP(t,k)$ so that the conclusions of Propositions \ref{prop:1}, \ref{prop:2} and \ref{prop:3} all hold simultaneously. From this fact we immediately obtain Theorem~\ref{theo:2} and, as a consequence, also {Theorems~\ref{theo:1}--\ref{theo:3} (see Section~\ref{sec:conclusion} and Section~\ref{sec:proofoftheo3} respectively for details).}

The key element in the proofs of Propositions \ref{prop:1}--\ref{prop:2} is the alternative representation in \eqref{eq:defhcirc2} of $\osbd$ as a directed LPP model, which allows us to effectively compare it with Brownian LPP. We now present this alternative representation in more detail.

Given the sequence $Y=(Y^{(r)})_{r \in \N}$ of independent Poisson processes $Y^{(r)}=(Y^{(r)}_t)_{t \geq 0}$ of rate $1$ used above to define the sets $\cV(t,k)$, we write $u \in Y^{(r)}$ to indicate that $u \in \R_{\geq 0}$ is a mark of $Y^{(r)}$, i.e., $Y^{(r)}_u \neq Y^{(r)}_{u^-}$, and, for $t > 0$ and $k \in \N$, consider the set
\begin{equation}\label{eq:defu}
\begin{split}
\cU(t,k):=\bigg\{ u \in \cV(t,k) : &\,  u(0)=1\,,\,\sup_{s \in (0,t]} |u(s)-u(s^-)| \leq 1\,,\\& \,u_r \in Y^{(r)} \text{ for all }r=1,\dots,u(t)-1\bigg\},
\end{split}
\end{equation} where, for each $r=0,\dots,k$, we write $u_r:=\inf \{ s \in [0,t] : u(s)> r\} \wedge t$ as before and, in addition, we stipulate the condition $u_r \in Y^{(r)}$ for all $r=1,\dots,u(t)-1$ is omitted whenever $u(t)=1$. In other words, $\cU(t,k)$ is the space of all c\` adl\`ag increasing paths starting at $1$ that have at most $k$ jumps, all with size $+1$, which can only jump from $r$ to $r+1$ at a given time $s$ if $s$ is a mark of $Y^{(r)}$. Observe that $\cU(t,k) \subseteq \cV(t,k)$ since paths in $\cV(t,k)$ are allowed to jump at arbitrary times while paths in $\cU(t,k)$ can only jump at Poisson marks. In particular, for any $u \in \cU(t,k)$ one can recover the value of $u(s)$ for all $s \in [0,t)$ from the vector $(u_0,\dots,u_k)$. Therefore, we may view elements of $\cU(t,k)$ either as paths as in~\eqref{eq:defu} or as vectors $(u_0,\dots,u_k) \in \R^{k+1}$ with $0=u_0 \leq u_1 \leq \dots \leq u_k = t$ such that, for every $r=1,\dots,k$, one has either $u_{r-1}=t$ or that $u_{r-1} < u_r$ and $u_r \in Y^{(r)} \cup \{ t\}$ (by the latter we mean that either $u_r \in Y^{(r)}$ or $u_r=t$). See Figure~\ref{fig:1-S} for an illustration. 

\begin{figure}
 \begin{center}
 \includegraphics[scale=0.58]{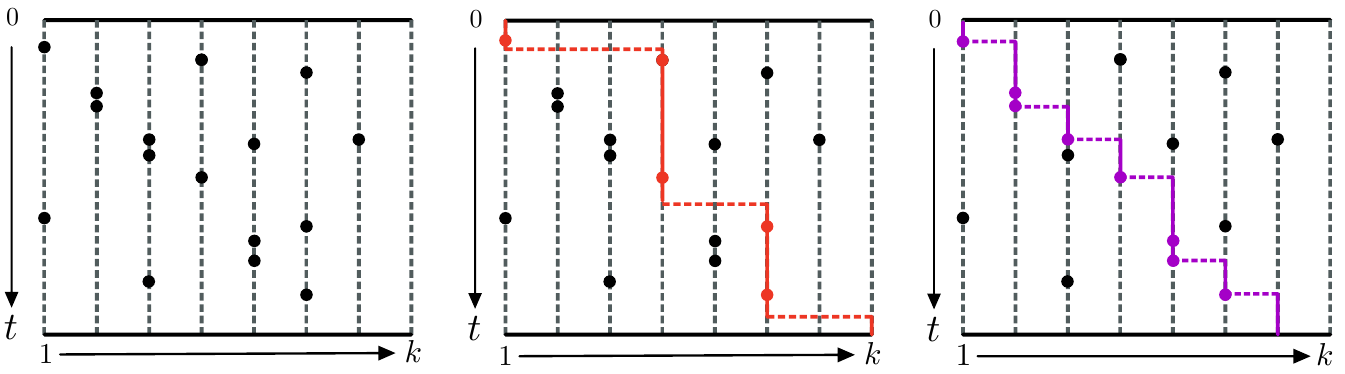}
 \caption{\textbf{Illustration of paths in $\cV(t,k)$ and $\cU(t,k)$}. The picture on the left shows a realization of the Poisson processes $Y^{(r)}$ on $[0,t]$ for  $r=1,\dots,k$. The path in red in the middle picture is an example of a path in $\cV(t,k)$ based on this realization, whereas the purple path on the rightmost picture is an example of path in $\cU(t,k)$. Red and purple dots in the second and third pictures respectively correspond to points accounted for by $H$ in \eqref{eq:defhy} and \eqref{eq:defhG}. We remark that the only difference between paths in $\cV(t,k)$ and $\cU(t,k)$ (and, consequently, in the definition of $\PLPP$ and $h^\circ_F$) is that, for $L$, paths are allowed to jump at arbitrary times, whereas, for $h^\circ_F$, paths are only allowed to jump at Poisson marks.}\label{fig:1-S}
\end{center}
\end{figure}

{Now, for any initial condition $G:\N \to [-\infty, \infty)$, the alternative representation of~$\osbd$ is defined as in \eqref{eq:defhcirc}, i.e., as the process $h^\circ_G:=(h^\circ_G(t,k) : t \geq 0\,,\,k \in \N)$ given by the formula
\begin{equation}
h^\circ_G(t,k):=\sup_{u \in \cU(t,k)} [H(u,Y)+G(k-u(t)+1)],
\end{equation} with $H$ as in \eqref{eq:defh}.
Then, we show that $h_G$ and $h_G^\circ$ have the same marginal distributions.}

\begin{lemma}\label{lemma:1} Let $h_G=(h_G(t,k) : t \geq 0\,,\,k \in \N)$ be a $\osbd$ process with initial condition $G:\N \to [-\infty,\infty)$ (not necessarily belonging to $\cI$). Then, given any $t > 0$ and $k \in \N$, {we have the equality in distribution
\begin{equation}\label{eq:defhG}
h_G(t,k)\overset{d}{=} h_G^\circ(t,k).
\end{equation} In particular, we have 
\[
h_S(t,k)\overset{d}{=} \sup_{u \in \cU(t,k) \atop u(t)=k} H(u,Y) 
\] with the convention that $\sup \emptyset:=-\infty$ and 
\[
h_F(t,k)\overset{d}{=} \sup_{u \in \cU(t,k)} H(u,Y).
\]}
\end{lemma}

On the other hand, to prove Proposition~\ref{prop:3}, we adapt a result due to Komlós, Major and Tusnády \cite{Komlos}, to obtain the following strong approximation of a Poisson process by Brownian motion with drift.

\begin{lemma}\label{lemma:2}
    There exist constants $C, \lambda, K, \kappa > 0$ such that, for any fixed $t \geq 1$, there exists a coupling of a Poisson process $(Y_s)_{s \geq 0}$ of rate $1$ and a standard Brownian Motion $(B_s)_{s \geq 0}$ such that, for all $x \geq C \log t$,
    \begin{equation}\label{eq:rwbound2}
        \P \lp \sup_{s \in [0,t]} \left| Y_s - s - B_s\right| >x \rp \leq K t^{\lambda} \rme^{- \kappa x}. 
    \end{equation}
\end{lemma}

Using some ideas from the proof of \cite[Theorem~1]{Bodineau}, we show that $|\PLPP(t,k) - t - \BLPP(t,k)|$ is bounded from above by a sum of $k$ i.i.d. random variables with tails given by~\eqref{eq:rwbound2}, from which Proposition~\ref{prop:3} follows by the exponential Tchebychev inequality.

\section{Outline of the proof of Theorem~\ref{theo:4}}\label{sec:outline2}

Note that, for any $s \in (0,1)$, we have the inclusion of events
\[
A^\gamma_G(t,k) \leq B^\gamma_G(t,k; s),
\] where $A_G^\gamma(t,k)$ and $B^\gamma_G(t,k;s)$ are defined as in \eqref{eq:defa} and \eqref{eq:defb}, respectively. In particular, from this it follows that
\begin{equation}\label{eq:ineqxi}
\xi_G(\alpha;s) \leq \xi_G(\alpha).
\end{equation} 
Therefore, in light of \eqref{eq:ineqxi}, Theorem~\ref{theo:4} will follow at once from the next two propositions.

\begin{proposition}
\label{prop:fluc1} Let $\alpha: \R_{\geq 0} \to \N$ verify $\displaystyle{\lim_{t \to \infty} \alpha(t)=\infty}$ and 
$\alpha(t)=o\big(t^{\frac{3}{7}}(\log t)^{-\frac{6}{7}}\big)$. Then, for any $G \in \cI$, $\gamma \in (0,2/3)$ and $s \in (0,1)$, we have
\[ 
\limsup_{t \to \infty} \bP(B^\gamma_G(t,\alpha(t);s)) < 1.
\] In particular, $\xi_G(\alpha;s)\geq \frac{2}{3}$ for all $s \in (0,1)$.  
\end{proposition}

{\begin{proposition}\label{prop:fluc2} Let $\alpha: \R_{\geq 0} \to \N$ verify $\lim_{t \to \infty} \frac{\alpha(t)}{(\log t)^\rho}=\infty$ for any $\rho > 0$ and $\alpha(t)=o(t^\eta)$ for some $\eta \in (0,\frac{9}{31})$. Then, for any $G \in \cI$ and $\gamma > \frac{2}{3}$ we have 
\begin{equation}\label{eq:asymp3}
\liminf_{t \to \infty} \P(A_G^\gamma(t,\alpha(t)))=1.
\end{equation} In particular, we have $\xi_G(\alpha) \leq \frac{2}{3}$ for all $\cG \in \cI$.
\end{proposition}}


We give the outline of the proof of each proposition in a separate subsection.


\refstepcounter{subsection} 
\subsection*{\thesubsection~Outline of the proof of the lower bound $\xi_G(\alpha;s) \geq \frac{2}{3}$}

To present the main ideas of the argument in a clear and concise fashion, we will assume throughout this outline that $s=\frac{1}{2}$ and $G=F$. The argument in the general case $s \in (0,1)$ and $G \in \cI$ is essentially the same, albeit with some (minor) additional technical difficulties, we refer the reader to Section~\ref{sec:proplbf} for details.

The general argument behind the proof of Proposition~\ref{prop:fluc1} in this case is as follows. Let $\overline{u}$ be any geodesic for $h^\circ_F(t,\alpha(t))$ and abbreviate $\overline{u}_*:=\overline{u}(\frac{1}{2}t)$. Then, by definition of geodesic, we have
\begin{equation}\label{eq:descompfacil}
h^\circ_F(t,\alpha(t)) = h^\circ_S(\tfrac{1}{2}t,\overline{u}_*) + h_F^\circ( (\tfrac{1}{2}t,\overline{u}_*) \to (t,\alpha(t))), 
\end{equation} where, for $\cG \in \cI$ and $(t_1,k_1)$, $(t_2,k_2) \in \R_{\geq 0} \times \N$ such that $t_1 \leq t_2$ and $k_1 \leq k_2$, we define
\begin{equation}\label{eq:defhpp}
\begin{split}
h_G^\circ((t_1,k_1) \to (t_2,k_2)) &= \sup_{u \in \cU(t_2,k_2) \atop u(t_1)=k_1} \left[\int_{t_1}^{t_2} \mathrm{d}Y^{(u(s))} + G(k_2-u(t_2)+1)\right] \\ &= \sup_{u \in \cU(t_2,k_2) \atop u(t_1)=k_1} \left[ \sum_{r=k_1}^{k_2} Y^{(r)}_{u_r} - Y^{(r)}_{\max\{ u_{r-1},t_1\}} + G(k_2-u(t_2)+1)\right], 
\end{split}
\end{equation} see {Figure~\ref{fig:C-2} for an illustration}.

\begin{figure}
 \begin{center}
 \includegraphics[scale=0.8]{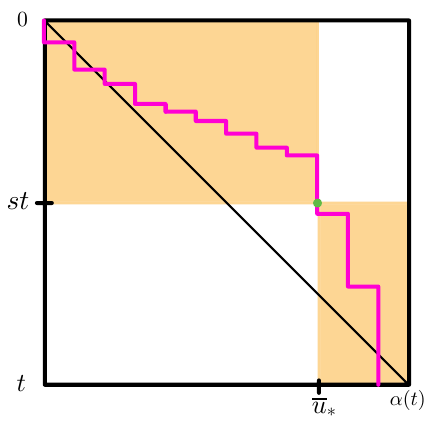}
 \caption{\textbf{Illustration of \eqref{eq:descompfacil} and \eqref{eq:decomph}.} The picture shows a geodesic $\overline{u}$ for $h^\circ_G(t,\alpha(t))$, colored in pink, and two orange rectangles, $R_1:=[0,st] \times [0,\overline{u}_*]$ and $R_2:=[st,t]\times [\overline{u}_*,\alpha(t)]$. The location $(st,\overline{u}_*)$ of $\overline{u}$ at time $t'=st$ is indicated by a green dot. The number of marks collected by $\overline{u}$ in $R_1$ and $R_2$ is respectively $h^\circ_S(st,\overline{u}_*)$ and $h^\circ_G((st,\overline{u}_*) \to (t,\alpha(t))$, giving immediately the equalities in \eqref{eq:descompfacil} and \eqref{eq:decomph}.}\label{fig:C-2}
\end{center}
\end{figure}

Therefore, if, for $t>0$ and $G \in \cI$, we write
\[
\overline{h^\circ_G}(t,\alpha(t)):=\frac{h^\circ_G(t,\alpha(t))-t-2\sqrt{t\alpha(t)}}{\sqrt{t(\alpha(t))^{-\frac{1}{3}}}}
\] and, for any $s \in (0,1)$ and $k \in [1,\alpha(t)]$, define the quantities
\begin{equation}\label{eq:defuk}
U_{t,\alpha}(s,k):=\frac{h^\circ_S(st,k) - st - 2\sqrt{stk}}{\sqrt{stk^{-\frac{1}{3}}}}
\end{equation} and
\begin{equation}\label{eq:defvk}
V_{t,\alpha}(s,k):=\frac{h^\circ_F( (st,k) \to (t,\alpha(t))) - (1-s)t - 2\sqrt{(1-s)t(\alpha(t)-k)}}{\sqrt{(1-s)t(\alpha(t)-k)^{-\frac{1}{3}}}},
\end{equation}
then, whenever $\overline{u}_*=\frac{1}{2}\alpha(t)$ (assuming $\frac{1}{2}\alpha(t) \in \N$ for simplicity), we can write
\begin{equation}\label{eq:decompshalf}
\overline{h^\circ_F}(t,\alpha(t)) =  \frac{1}{2^{1/3}}\left(U_{t,\alpha}(\tfrac{1}{2},\overline{u}_*) +V_{t,\alpha}(\tfrac{1}{2},\overline{u}_*)\right).
\end{equation} Now, on the one hand, $\overline{h^\circ_F}(t,\alpha(t))$ and $U_{t,\alpha}(\tfrac{1}{2},\tfrac{1}{2}\alpha(t))$ converge in distribution to~$\TW_{\GUE}$ by virtue of Theorem~\ref{theo:1} and Lemma~\ref{lemma:1}. On the other hand, by translation invariance of the sequence $(Y^{(r)})_{r \in \N}$ (when viewed as a single Poisson process on $\R_{\geq 0} \times \N$), we~have
\begin{equation}\label{eq:decompshalf2}
h^\circ_F\left((\tfrac{1}{2}t,\tfrac{1}{2}\alpha(t)) \to (t,\alpha(t))\right) \overset{d}{=} h^\circ_F\left(\tfrac{1}{2}t,\tfrac{1}{2}\alpha(t)+1\right)
\end{equation} so that, by Theorem~\ref{theo:1} and Lemma~\ref{lemma:1} again, $V_{t,\alpha}(\tfrac{1}{2},\tfrac{1}{2}\alpha(t))$ also converges to~$\TW_{\GUE}$. Moreover, $U_{t,\alpha}(\tfrac{1}{2}t,\tfrac{1}{2}\alpha(t))$ and $V_{t,\alpha}(\tfrac{1}{2}t,\tfrac{1}{2}\alpha(t))$ are independent since they are functions of non-overlapping regions of the Poisson process $(Y^{(r)})_{r \in \N}$. In light of these facts, \eqref{eq:decompshalf} immediately implies that the event $\{\overline{u}_*=\tfrac{1}{2}\alpha(t)\}$ cannot occur with probability tending to $1$ as $t \to \infty$. Indeed, if this were the case, then from \eqref{eq:decompshalf} and the above discussion we would obtain that
\begin{equation}\label{eq:decompshalf3}
\TW_{\GUE} = \frac{1}{2^{1/3}} ( \TW_{\GUE} \oplus \TW_{\GUE})
\end{equation} (where $\oplus$ indicates that the summands on the right-hand side are independent), a fact which is not true: for example, the variances on both sides of \eqref{eq:decompshalf3} do not match.

We now claim that a similar argument can be used to show that, for $\gamma < \frac{2}{3}$, the event $\{|\overline{u}_*-\tfrac{1}{2}\alpha(t)|\leq (\alpha(t))^\gamma\}$ cannot occur with probability tending to $1$ as $t \to \infty$. Indeed, this will follow from the key fact (which needs to be shown) that, for any fixed $s \in (0,1)$, if $C^\gamma(t,\alpha(t);s)=\{ k \in \N: |k-s\alpha(t)|\leq (\alpha(t))^\gamma\}$ denotes the cross section of the cylinder $C^\gamma(t,\alpha(t))$ at time $t'=st$, then the random variables in each of the collections
\[
\cC_{t,\alpha,\gamma}^{(1)}(s):=\left(U_{t,\alpha}(s,k) : k \in C^\gamma(t,\alpha(t);s)\right)
\] \[
\cC_{t,\alpha,\gamma}^{(2)}(s):=\left(V_{t,\alpha}(s,k) : k \in C^\gamma(t,\alpha(t);s)\right),
\] are \textit{strongly correlated} if $\gamma < \frac{2}{3}$ and  $t$ is large enough, which will allow~us to approximate the height $h^\circ_F(t,\alpha(t))$ on the event $\{|\overline{u}_* - \tfrac{1}{2}\alpha(t)|\leq (\alpha(t))^\gamma\}$ by
\[
\begin{split}
h^\circ_F(t,\alpha(t)) & = h^\circ_S(\tfrac{1}{2}t,\overline{u}_*) + h_F^\circ( (\tfrac{1}{2}t,\overline{u}_*) \to (t,\alpha(t))) \\
& \approx 
h_S^\circ ( \tfrac{1}{2}t,\tfrac{1}{2}\alpha(t)) + h_F^\circ( (\tfrac{1}{2}t,\tfrac{1}{2}\alpha(t)) \to (t,\alpha(t))),
\end{split}
\] and then essentially reproduce the same argument given above to conclude the result. More precisely, the strong correlation result we will show is the following (analogous to the one proved in \cite{Basu} for Brownian LPP), which states that, for $i=1,2$, the maximum of the random variables in  $\mathcal{C}^{(i)}_{t,\alpha,\gamma}(s)$ has the same tail behavior (in the limit as $t \to \infty$) as each single variable in $\mathcal{C}^{(i)}_{t,\alpha,\gamma}(s)$.

\begin{lemma}\label{lemma:fluc1a} Let $\alpha: \R_{\geq 0} \to \N$ verify $\lim_{t \to \infty} \alpha(t)=\infty$ and $\alpha(t)=o\big(t^{\frac{3}{7}}(\log t)^{-\frac{6}{7}}\big)$. Then, given any $\gamma \in (0,2/3)$, $s \in (0,1)$ and $\varepsilon > 0$, there exist $t_{\gamma,s,\varepsilon},x_{\gamma,s,\varepsilon} >0$ such that, for all $t> t_{\gamma,s,\varepsilon}$ and $x \in [x_{\gamma,s,\varepsilon},\min\{ (\alpha(t))^{1/20},(\log t)^{1/4}\}]$,
\begin{equation}\label{eq:maxleft}
\rme^{-\frac{1}{12}(1+\varepsilon)x^{3}} \leq \bP\left( \max_{k \in C^\gamma(t,\alpha(t);s)} U_{t,\alpha}(s,k) \leq -x\right) \leq \rme^{-\frac{1}{12}(1-\varepsilon)x^{3}}
\end{equation} and
\begin{equation}\label{eq:maxright}
\mathrm{e}^{-\frac{4}{3}(1+\varepsilon)x^{3/2}}\leq \P\left(\max_{k \in C^\gamma(t,\alpha(t);s)} U_{t,\alpha}(s,k) \geq x\right) \leq \mathrm{e}^{-\frac{4}{3}(1-\varepsilon)x^{3/2}}.
\end{equation} Moreover, the same tail estimates hold also for $\max_{k \in C^\gamma(t,\alpha(t);s)} V_{t,\alpha}(s,k)$. 
\end{lemma}


Nonetheless, we point out that the argument given above whenever $\overline{u}_*=\frac{1}{2}\alpha(t)$ does not immediately carry over to the case in which one only has that $|\overline{u}_* - \frac{1}{2}\alpha(t)|\leq (\alpha(t))^\gamma$. Indeed, on the one hand, the equality in \eqref{eq:decompshalf} is not true anymore if $\overline{u}_*\neq \tfrac{1}{2}\alpha(t)$, we only have the $\leq$ inequality. In addition, the dependence in $u_*$ on the right-hand side of \eqref{eq:decompshalf} may affect the convergence to $\frac{1}{2^{1/3}}(\TW_{\GUE} \oplus \TW_{\GUE})$. To handle both issues at once, we will replace the random variables on the right-hand side of \eqref{eq:decompshalf} by the maximum over all random variables in $\mathcal{C}^{(1)}_{t,\alpha,\gamma}(\tfrac{1}{2})$ and $\cC^{(2)}_{t,\alpha,\gamma}(\tfrac{1}{2})$, respectively, which exhibit the same asymptotic tail behavior as the original variables by Lemma~\ref{lemma:fluc1a}. More precisely, we will show that, on the event $\{|\overline{u}_* - \tfrac{1}{2}\alpha(t)| \leq (\alpha(t))^\gamma\}$, we have
\begin{equation}\label{eq:decompshalf4}
\overline{h^\circ_F}(t,\alpha(t)) \leq \frac{1}{2^{1/3}}\left[ \max_{k \in C^\gamma(t,\alpha(t);\frac{1}{2})} U_{t,\alpha}(\tfrac{1}{2},k) + \max_{k \in C^\gamma(t,\alpha(t);\frac{1}{2})} V_{t,\alpha}(\tfrac{1}{2},k)\right]+o(1),
\end{equation} where the term $o(1)$ converges to $0$ in probability as $t \to \infty$. However, these maxima are no longer guaranteed to converge to $\TW_{\GUE}$ (even if asymptotically they do have the same tail behavior), and thus, to conclude that \eqref{eq:decompshalf4} cannot occur with overwhelming probability as $t \to \infty$, we will not be able to argue as before by comparing the variances of the limiting distributions on both sides of the inequality. Instead, what we will do is show that both sides have (in the limit as $t \to \infty$) a different left tail, as indicated by the next result.

\begin{lemma}\label{lemma:fluc1b} Let $\beta:\R_{\geq 0} \to \R_{\geq 0}$ be such that $\lim_{t \to \infty} \beta(t)=\infty$ and suppose that, for each $t > 0$, $U_t$ and $V_t$ are independent random variables satisfying that, given $\varepsilon > 0$, there exist $t^*_{\varepsilon} >0$ and $x^*_{\varepsilon} \geq 1$ such that, if $t > t^*_\varepsilon$ and $x \in [x^*_{\varepsilon},\beta(t)]$, 
\[
\rme^{-\frac{1}{12}(1+\varepsilon)x^{3}} \leq \bP\left( U_t \leq -x\right) \leq \rme^{-\frac{1}{12}(1-\varepsilon)x^{3}}
\] and 
\[
\rme^{-\frac{1}{12}(1+\varepsilon)x^{3}} \leq \bP\left( V_t \leq -x\right) \leq \rme^{-\frac{1}{12}(1-\varepsilon)x^{3}}.
\] Then, given $s \in (0,1)$, if $\varepsilon$ is chosen small enough, for any $\tilde{x}\geq x_\varepsilon^*$ there exists $t^*_{\varepsilon,s,\tilde{x}} >0$ such that, for all $t > t^*_{\varepsilon,s,\tilde{x}}$, we have
\begin{equation}\label{eq:cotaint3}
\P\left(s^{\frac{1}{3}}U_{t} + (1-s)^{\frac{1}{3}}V_{t} \leq -z_{\varepsilon,s}(\tilde{x})\right) \geq \frac{1}{2} \mathrm{e}^{-\frac{1}{12}(1+\varepsilon)c_s(z_{\varepsilon,s}(\tilde{x}))^3},
\end{equation} where $z_{\varepsilon,s}(\tilde{x}):=(s^{1/3}+(1-s)^{1/3})\frac{1}{\varepsilon}\tilde{x}$ and $c_s:=\frac{1}{1+2\sqrt{s(1-s)}} < 1$.
\end{lemma}

The argument in full detail (as well as the proof for general $s \in (0,1)$ and $G \in \cI$) can be found in Section~\ref{sec:proplbf}.

\refstepcounter{subsection} 
\subsection*{\thesubsection~Outline of the proof of the upper bound $\xi_G(\alpha) \leq \frac{2}{3}$}

The argument of the proof borrows some ideas from \cite[Section 3]{JohanssonTransFluc}. Fix some $\gamma > \frac{2}{3}$ and suppose that $\overline{u}$ is a geodesic for $h^\circ_G(t,\alpha(t))$ whose graph is not entirely contained in $C^\gamma(t,\alpha(t))$. For the purpose of conveying the main idea of the argument more clearly, let us assume that $\overline{u}$ first exits $C^\gamma(t,\alpha(t))$ at time $t'=\tfrac{1}{2}t$ and that it does so from above, i.e., $\overline{u}(\tfrac{1}{2}t)=\lceil \tfrac{1}{2}\alpha(t) + (\alpha(t))^\gamma \rceil$, see Figure~\ref{fig:C-1} for an illustration. Then, if we write $k_{t,\alpha}^\gamma:=\lceil \tfrac{1}{2}\alpha(t) + (\alpha(t))^\gamma \rceil$ for simplicity, by definition of geodesic we have
\[
h^\circ_G(t,\alpha(t))= h^\circ_S(\tfrac{1}{2}t,k_{t,\alpha}^\gamma) + h_G^\circ\left((\tfrac{1}{2}t,k_{t,\alpha}^\gamma) \to (t,\alpha(t))\right),
\] for    $h_G^\circ\left((\tfrac{1}{2}t,k_{t,\alpha}^\gamma) \to (t,\alpha(t))\right)$ defined as in \eqref{eq:defhpp}. In particular, upon recalling \eqref{eq:defuk}--\eqref{eq:defvk}, the former equality implies that
\[
\frac{h^\circ_G(t,\alpha(t))-t-2\sqrt{t\alpha(t))}}{\sqrt{t(\alpha(t))^{-\frac{1}{3}}}} = \frac{1}{2^{1/3}}\left(U_{t,\alpha}(\tfrac{1}{2},k_{t,\alpha}^\gamma) +V_{t,\alpha}(\tfrac{1}{2},k_{t,\alpha}^\gamma)\right)+2R(t),
\] where 
\[
R(t)=\frac{\sqrt{(\tfrac{1}{2}t)k_{t,\alpha}^\gamma} + \sqrt{(\tfrac{1}{2}t)(\alpha(t)-k_{t,\alpha}^\gamma)} - \sqrt{t\alpha(t)}}{\sqrt{t(\alpha(t))^{-\frac{1}{3}}}}+o(1)
\] with the term $o(1)$ 
converging to $0$ in probability as $t \to \infty$. 
Now, a straightforward computation shows that, if in addition to the condition $\gamma > \frac{2}{3}$ we ask that $\gamma < 1$ (which, for the purpose of proving the result, we can do without losing generality), then 
\[
\sqrt{(\tfrac{1}{2}t)k_{t,\alpha}^\gamma} + \sqrt{(\tfrac{1}{2}t)(\alpha(t)-k_{t,\alpha}^\gamma)} - \sqrt{t\alpha(t)} \leq - \frac{1}{64}\sqrt{t(\alpha(t))^{-\frac{1}{3}}}(\alpha(t))^{2(\gamma-1)},
\] for all $t$ large enough. Indeed, this is (essentially) a consequence of the following lemma, which we shall need in its full generality to treat the case $t' \neq \frac{1}{2}$ in Section~\ref{sec:propubf}.

\begin{lemma}
\label{lemma_geodesics_2-0} Let $\alpha: \R_{\geq 0} \to \N$ satisfy $\lim_{t \to \infty} \alpha(t)=\infty$. Then, given any $\gamma \in (\frac{2}{3},1)$, if $t$ is taken sufficiently large (depending only on $\alpha$ and $\gamma$), for all $s \in [0,t(1-(\alpha(t))^{\gamma-1}]$, we have
\[
\sqrt{sk_{t,\alpha}^\gamma(s)} + \sqrt{(t- s)(\alpha(t) - k_{t,\alpha}^\gamma(s))} - \sqrt{t\alpha(t)} \leq - \frac{1}{32}\sqrt{t(\alpha(t))^{-\frac{1}{3}}} (\alpha(t))^{2(\gamma-\frac{2}{3})}.
\] where $k_{t,\alpha}^\gamma(s):=\lfloor \tfrac{s}{t}\alpha(t) + (\alpha(t))^\gamma \rfloor$.
\end{lemma}

In particular, we conclude that $R(t) \leq -\frac{1}{64}(\alpha(t))^{2(\gamma-\frac{2}{3})}+o(1)$ for all $t$ large enough and hence, since $\gamma > \frac{2}{3}$, that $\lim_{t \to \infty} R(t)=-\infty$ in probability. Finally, since 
\[
\frac{1}{2^{1/3}}\left(U_{t,\alpha}(\tfrac{1}{2},k_{t,\alpha}^\gamma) +V_{t,\alpha}(\tfrac{1}{2},k_{t,\alpha}^\gamma)\right) \overset{d}{\longrightarrow} \frac{1}{2^{1/3}} ( \TW_{\GUE} \oplus \TW_{\GUE})
\] by an argument similar to the one discussed during the previous subsection, the former implies that, unless the event $\overline{u}(\tfrac{1}{2}t)=k_{t,\alpha}^\gamma$ occurs with vanishing probability as $t \to \infty$, the sequence
\[
\left(\frac{h^\circ_G(t,\alpha(t))-t-2\sqrt{t\alpha(t))}}{\sqrt{t(\alpha(t))^{-\frac{1}{3}}}} : t > 1\right)
\] cannot be tight, but this, in turn, contradicts the fact that it must converge to $\TW_{\GUE}$ by Theorem~\ref{theo:1}. Therefore, we conclude that the probability that any geodesic $\overline{u}$ exits $C^\gamma(t,\alpha(t))$ from above at time $t'=\tfrac{1}{2}t$ must vanish as $t \to \infty$. An analogous argument shows that the probability of exiting from below at time $t'=\tfrac{1}{2}t$ must also vanish. {Finally, to treat the general case in which the first exit time from $C^\gamma(t,\alpha(t))$ is some random time $t' \in (0,1]$, we subdivide the interval $[0,t]$ into small subintervals of (approximate) length $\frac{1}{t}$ and then repeat the previous analysis depending on which of these intervals this first exit time $t'$ is located. To conclude the argument, we will need to perform an union bound over all such subintervals and use the moderate deviation estimates from Theorem~\ref{theo:3} to deduce that the event that any geodesic exits $C^\gamma(t,\alpha(t))$ has vanishing probability as $t \to \infty$. However, when attempting to do this, we will face two issues (which are not present in \cite{JohanssonTransFluc}):
\begin{enumerate}
\item [i.] Since our moderate deviation estimates for $h_G(t,k)$ hold only for points $(t,k)$ which are close to the $t$-axis, to be able to apply these estimates successfully to each of the subintervals of the partition of $[0,t]$, we will need our curve $\alpha$ not to grow too fast: at this point is where the condition $\alpha(t)=o(t^\eta)$ for some $\eta \in (0,\frac{9}{31})$ is required. 
\item [ii.] For the union bound over all the subintervals in the partition of $[0,t]$ to be successful, we will need the moderate deviation estimates to decay fast enough: it is at this point the condition $\lim_{t \to \infty} \frac{\alpha(t)}{(\log t)^\rho}=\infty$ for all $\rho > 0$ is required.
\end{enumerate} We leave the full details of the proof to Section~\ref{sec:propubf}.}

\section{Preliminaries for the proof of Theorem~\ref{theo:1}}
\label{preliminaries.long.fluct}

In this section we prove all preliminary results needed to establish Propositions~\ref{prop:1}, \ref{prop:2} and \ref{prop:3} found in  Section~\ref{sec:outline}. Namely, we prove Lemma~\ref{lemma:1} and Lemma~\ref{lemma:2}, as well as state and prove an estimate on the tails of $\PLPP$, contained in Lemma~\ref{lemma:3} below.

\subsection{Proof of Lemma~\ref{lemma:1}}

Recall the graphical construction of the process $h_G$ in \eqref{eq:definition.osbd} in terms of the Poisson processes $(Q^{(r)})_{r \in \N}$. Let us fix any $t> 0$ and $k \in \N$ and define, for each $r=1,\dots,k$, the process $Y^{(r),t,k}=(Y^{(r),t,k}_s)_{s \in [0,t]}$ by the formula
\begin{equation}\label{eq:defy}
Y^{(r),t,k}_s:=Q^{(k-(r-1))}_{t^-} - Q^{(k-(r-1))}_{(t-s)^-}=|\{ u \in [t-s,t) : u \in Q^{(k-(r-1))}\}|,
\end{equation} where $Q^{(k-(r-1))}_{0^-}:=0$ by convention. Now, by time reversibility of the Poisson process, we have
\[
(Y^{(r),t,k}_s : s \in [0,t]\,,\,r=1,\dots,k)\overset{d}{=} (Q^{(r)}_s : s \in [0,t]\,,\,r=1,\dots,k),
\] so that, in particular, to obtain Lemma~\ref{lemma:1} it suffices to show, for any $k \in \N$ and $t > 0$, that
\begin{equation}
    \label{eq:rep1}
h_G(t^-,k)=\sup_{u \in \cU(t,k)} \left[ \sum_{r=1}^{u(t)} (Y^{(r),t,k}_{u_r} - Y^{(r),t,k}_{u_{r-1}}) + G(k-u(t)+1)\right],
\end{equation} for $\cU(t,k)$ defined as in \eqref{eq:defu} with respect to the sequence $(Y^{(r),t,k})_{r=1,\dots,k}$ given by~\eqref{eq:defy}, 
see Figure~\ref{fig:graph_const} for an illustration.

\begin{figure}
    \centering
    \includegraphics[width=0.4\linewidth]{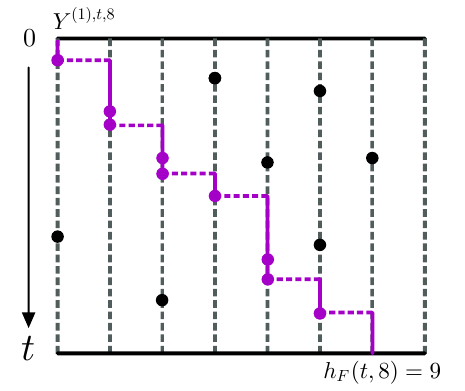}
\includegraphics[width=0.45\linewidth]{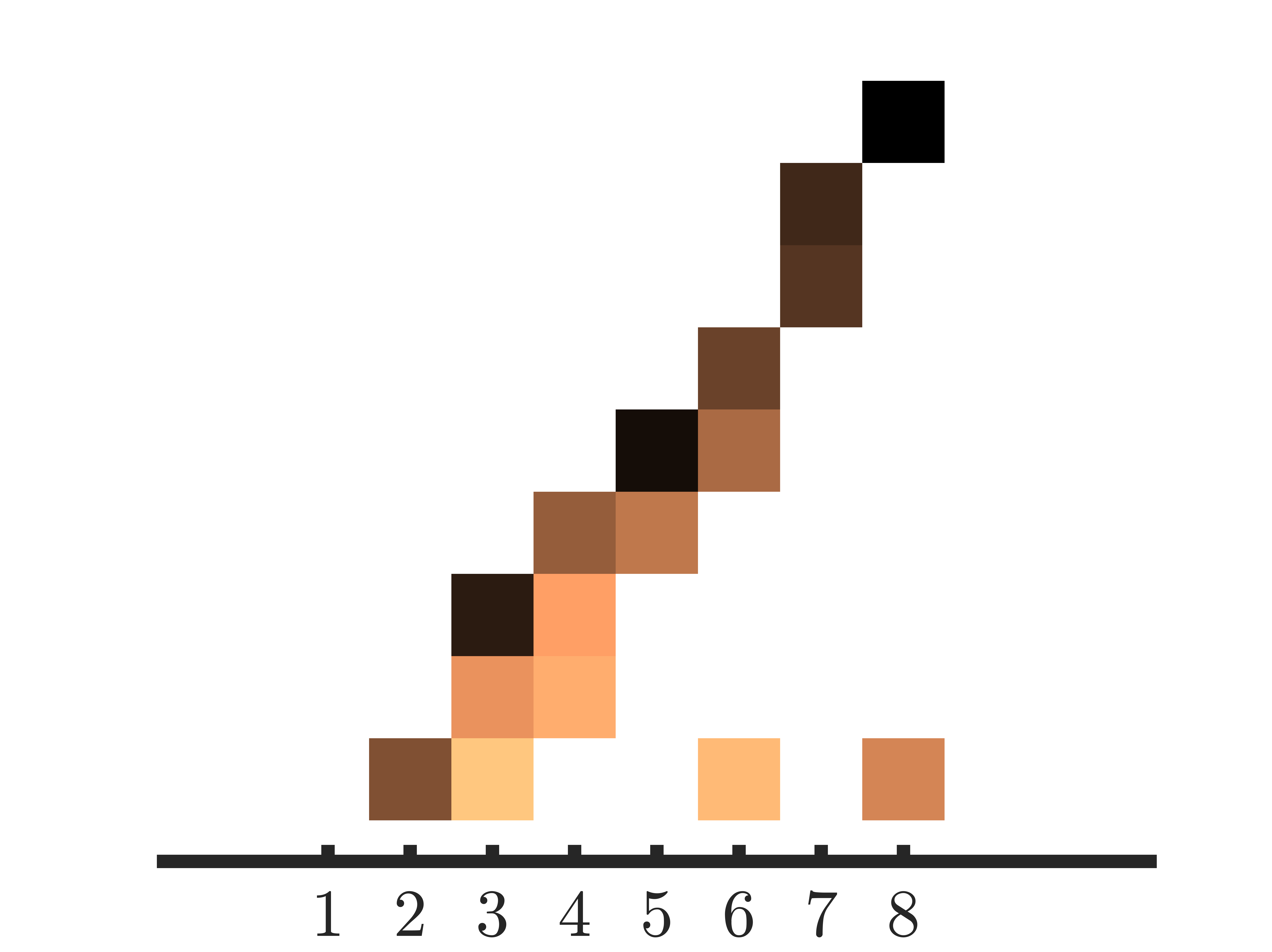}
    \caption{A realization of the Poisson process $Y$ constructed from $Q$ by reversing time and space (left) and the interface determined by this sample (right). $Q$-marks correspond to the following columns (ordered increasingly in time): 3, 6, 4, 4, 3, 8, 5, 6, 4, 2, 6, 7, 7, 3, 5, 8. The height $h_F(t,8)$ at site 8 is 9 since there is a path in $\cU(t,8)$ that collects 9 marks (and no path collects more than 9 marks). }
    \label{fig:graph_const}
\end{figure}

We proceed by induction on $k$. If $k=1$, then, since by definition $\cU(t,1)$ contains only the path $\overline{u}$ constantly equal to $1$ on $[0,t]$ for which we clearly have $\overline{u}_0=0$ and $\overline{u}_1=t$, the right-hand side of \eqref{eq:rep1} in this case is equal to
\[
Y^{(1),t,1}_t - Y^{(1),t,1}_0 + G(1) = Q^{(1)}_{t^-} + G(1) =h_G(t^-,1)
\] and so \eqref{eq:rep1} holds.

Now, assume that \eqref{eq:rep1} holds for some fixed $k \in \N$ and all $t > 0$. We~wish to check that then the analogous statement holds for $k+1$. To this end, consider an enumeration $q_1 < q_2 < \dots$ of the discontinuity points of $Q^{(k+1)}$ and let $\tau_t :=\sup \{ j ; q_j < t\}$ denote the index of the last discontinuity of $Q^{(k+1)}$ in $[0,t)$, with the convention $\sup \emptyset := 0$, so that $\tau_t=0$ if $Q^{(k+1)}$ has no discontinuities in $[0,t)$. Then, by definition of $h_G$, 
\[
h_G(t^-,k+1)=h(q_{\tau_t},k+1)=\begin{cases}G(k+1) & \text{ if $\tau_t=0$}\\ 1+ \max\{h_G(q_{\tau_t}^-,k),h_G(q_{\tau_t}^-,k+1)\} & \text{ if $\tau_t > 0$.}\end{cases}
\] In order to check that \eqref{eq:rep1} holds for $k+1$ and all $t > 0$, we may proceed by induction on the value~of~$\tau_t$. Indeed, if $\tau_t=0$, then $\cU(t,k+1)$ again consists only of the path $\overline{u}$ which is constantly equal to $1$ and $Q^{(k+1)}_{t^-}=0$, so that the right-hand side of~\eqref{eq:rep1} (with $k+1$ in place of $k$) becomes 
\[
Y^{(1),t,k+1}_t-Y^{(1),t,k+1}_0+G(k+1)=Q^{(k+1)}_{t^-}+G(k+1)=h_G(t^-,k+1).
\] Now, if \eqref{eq:rep1} holds for all $t > 0$ with $\tau_t = j$ for some $j \in \N_0$, then for any $s > 0$ with $\tau_s =j+1$ we have 
\begin{equation}\label{eq:rep5}
h_G(s^-,k+1)=1+ \max\{h_G(q_{j+1}^-,k),h_G(q_{j+1}^-,k+1)\}
\end{equation} Observe that, if we write $u_0^*:=0$ and $u^*_1:=s-q_{j+1}$, then we have
\begin{equation}\label{eq:rep2}
Y^{(1),s,k+1}_{u_1^*}-Y^{(1),s,k+1}_{u_0^*}=Q^{(k+1)}_{s^-}-Q^{(k+1)}_{q_{j+1}^-}=1.
\end{equation} On the other hand, since $\tau_{q_{j+1}}=j$, by induction hypothesis we have
\begin{equation}\label{eq:rep3}
h_G(q_{j+1}^-,k)=\sup_{u \in \cU(q_{j+1},k)} \left[ \sum_{r=1}^{u(t)} (Y^{(r),q_{j+1},k}_{u_r} - Y^{(r),q_{j+1},k}_{u_{r-1}}) + G(k-u(t)+1)\right],
\end{equation} and
\begin{equation}\label{eq:rep4}
h_G(q_{j+1}^-,k+1)=\sup_{u \in \cU(q_{j+1},k+1)} \left[ \sum_{r=1}^{u(t)} (Y^{(r),q_{j+1},k+1}_{u_r} - Y^{(r),q_{j+1},k+1}_{u_{r-1}}) + G(k-u(t)+2)\right].
\end{equation} Upon noticing that
\[
Y^{(r),q_{j+1},k+1-l}_{u} - Y^{(r),q_{j+1},k+1-l}_{u'} = Y^{(r+l),s,k+1}_{u+u_1^*} - Y^{(r+l),s,k+1}_{u'+u_1^*} 
\] for all $u,u' \in [0,q_{j+1}]$, $l \in \{0,1\}$ and $r=1,\dots,k+1-l$, and also that
\begin{equation}\label{eq:split}
\begin{split}
\cU(s,k+1)=&\,\{ (0,u_1^*,u_1^*+u_1,\dots,u_1^*+u_m) : (u_0,\dots,u_m) \in \cU(q_{j+1},k)\}\\ & \cup \{ (0,u_1^*+u_1,\dots,u_1^*+u_m) : (u_0,\dots,u_m) \in \cU(q_{j+1},k+1)\},
\end{split}  
\end{equation} where above we identify each path in $u \in \cU(q,r)$ with its associated vector $(u_0,\dots,u_r)$ (see Figure~\ref{fig:2-S} for an illustration of the equality in~\eqref{eq:split}), by combining \eqref{eq:rep2}--\eqref{eq:rep3}--\eqref{eq:rep4} it is straightforward to check that the right-hand side of \eqref{eq:rep5} equals
\[
\sup_{u \in \cU(s,k+1)} \left[ \sum_{r=1}^{u(t)} (Y^{(r),s,k+1}_{u_r} - Y^{(r),s,k+1}_{u_{r-1}}) + G(k+1-u(t)+1)\right],
\] so that \eqref{eq:rep1} follows (for $k+1$ and all $s > 0$ with $\tau_s=j+1$). This concludes the proof.

\begin{figure} 
\begin{center}
\includegraphics[scale=.75]{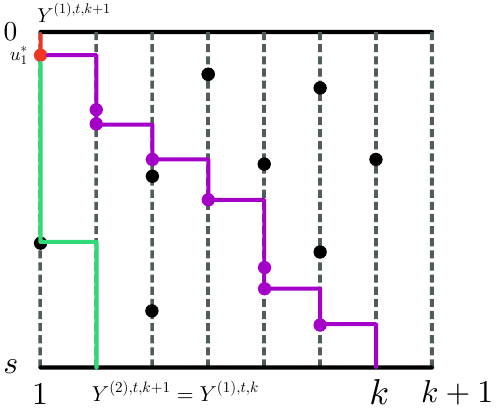}
\caption{\textbf{Equality in \eqref{eq:split}}. The picture shows the two types of paths in $\cU(s,k+1)$: a path can either (1) jump at $u_1^*$, the first Poisson point it encounters, and then continue from the next column onward as a path in $\cU(q_{j+1}, k)$; or (2) choose not to jump at $u_1^*$ and then continue from the first column onward as a path in $\cU(q_{k+1},k+1)$. The first option is depicted as a purple path, the second as a green path. The common part of both paths is colored in red.}\label{fig:2-S}
\end{center}
\end{figure}

\subsection{Proof of Lemma~\ref{lemma:2}}

Below we restate the strong approximation theorem by Komlós, Major and Tusnády adapted to our situation. Since the Poisson distribution has exponential moments of all orders, it reads as follows.

\begin{theorem}[Theorem~1 in \cite{Komlos}]
\label{theo:RWbound}
Let $(X_i)_{i \in \N}$ be i.i.d. Poisson random variables with $\E(X_i) =1$. Then, there exist constants $C_1, K_1, \lambda_1 >0$ such that, for each $N \in \N$, there exists a coupling of $(X_1,...,X_N)$ with a standard Brownian Motion $(B_t)_{t \geq 0}$ such that, for all $x \geq C_1 \log N$,
\begin{equation}\label{eq:pbrwbound}
    \P \lp \max_{n = 1,...,N} \left| \sum_{i = 1}^{n} X_i - n - B_n \right| >x \rp \leq K_1 N^{\lambda_1} \rme^{-\lambda_1 x}.
\end{equation}
\end{theorem}

Theorem~\ref{theo:RWbound} can now be used to prove Lemma~\ref{lemma:2} as follows.


\begin{proof}[Proof of Lemma~\ref{lemma:2}]
    Let us fix any $t \geq 1$ and define $N:=\lfloor t \rfloor$. By Theorem~ \ref{theo:RWbound} there exists a coupling of $N$ i.i.d. Poisson random variables $X_1,...,X_N$ of mean $1$ and a standard Brownian motion $(B_s)_{s \geq 0}$ so that \eqref{eq:pbrwbound} holds for all $x \geq C_1 \log N$. By enlarging this probability space if necessary, we may assume also that in the latter there is also defined a Poisson process $(Y_s)_{s \geq 0}$ of rate $1$ such that $Y_{n}-Y_{n-1}=X_n$ for all $n=1,\dots,N$. Indeed, it is enough to extend $(X_1,\dots,X_N)$ to a sequence $(X_n)_{n \in \N}$ of i.i.d. Poisson random variables of mean $1$, consider an array $(U^{(n)}_i)_{i,n \in \N}$ of i.i.d. random variables uniformly distributed on $(0,1)$ independent of $(X_n)_{n \in \N}$, set $X_0:=0$ and then, for $s \geq 0$, define
    \[
    Y_s :=\sum_{i=0}^{\lfloor s \rfloor} X_i + \sum_{i=1}^{X_{\lceil s \rceil}} \mathbf{1}_{\{U^{(\lceil s \rceil)}_i \leq s-\lfloor s \rfloor\}}. 
    \]
    Next, fix $x \geq 2C_1 \log N$. Then, if we let $A^{(N)}_x$ denote the event on the left-hand side of \eqref{eq:pbrwbound} and we write $Z_s:=Y_s-s$ for each $s\geq 0$, then by the union bound we obtain that the probability on the left-hand side of \eqref{eq:rwbound2} is bounded from above by
    \begin{equation}\label{eq:rwbound3}
    \sum_{n=1}^{N+1} \P\lp \bigg\{ \sup_{s \in [n-1,n]} |Z_s-B_s| > x\bigg\} \setminus A^{(N)}_{\frac{x}{2}}\rp + \P( A^{(N)}_{\frac{x}{2}}). 
    \end{equation} The rightmost probability in this last display is bounded from above by $K_1N^{\lambda_1}\rme^{-\frac{\lambda_1}{2}x}$ by Theorem~\ref{theo:RWbound}. Hence, it is enough to bound each summand in the sum on the left. To this end, notice that, for each $n=1,\dots,N+1$, the $n$-th summand in this sum is bounded from above by
\[
\begin{split}
   \P\Bigg( \bigg\{ \sup_{s \in [n-1,n]} |Z_s-B_s & - (Z_{n-1}-B_{n-1})| > \frac{x}{2}\bigg\} \setminus A^{(N)}_{\frac{x}{2}}\Bigg) \\
        & \leq \P\lp \sup_{s \in [n-1,n]} |Z_s- Z_{n-1}|> \frac{x}{4}\rp+\P\lp \sup_{s \in [n-1,n]} |B_s-B_{n-1}| > \frac{x}{4}\rp \\
        & \leq \P\lp \sup_{s \in [0,1]} |Z_s|> \frac{x}{4}\rp+\P\lp \sup_{s \in [0,1]} |B_s| > \frac{x}{4}\rp 
        \leq c\rme^{-\frac{x}{4}},
    \end{split} 
\] for some $c> 0$, where the last inequality follows from Doob's maximal inequality applied to the nonnegative submartingales $(\rme^{|Z_s|})_{s\geq 0}$ and $(\rme^{|B_s|})_{s \geq 0}$, while the second to last one does so from the Markov property. Upon summing this bound over all $n=1,\dots,N+1$ we conclude that \eqref{eq:rwbound3} is bounded by $Kt^\lambda \rme^{-\kappa x}$ for all $x \geq C\log t$, where $C:=2C_1$, $\lambda:=\max\{\lambda_1,1\}$, $K:=2c+K_1$ and $\kappa:=\min\{\frac{\lambda_1}{2},\frac{1}{4}\}$ and so \eqref{eq:rwbound2} holds.
\end{proof}

\subsection{Tail estimates for $\PLPP$}

To prove Propositions~\ref{prop:1} and \ref{prop:2}, we need the following estimate on the tail of the distribution of $\PLPP(t,k)$, which will prove most useful when considering values of $k$ which are not much smaller than $t$.
\begin{lemma} \label{lemma:3}
    There exists a constant $C > 0$ such that, for any $t > 0$ and $x,k \in \N$ with $x \geq t$, we have 
    \begin{equation}\label{eq:boundD}
        \P \lp \PLPP (t,k) \geq x \rp
        \leq C\exp \left[ x \lp \log \lp \frac{(k+x)t}{x^2} \rp +2\rp -t\right].
    \end{equation}
\end{lemma}

\begin{proof} We follow the argument in \cite{Penrose}. Given $x,k \in \N$, consider the set 
\[
\cN_{x,k} := \ls y=(y_1,...,y_k) \in \N_0^k : \sum_{r= 1}^k y_r =x\rs,
\] and, for $y \in \cN_{x,k}$, define a collection $(W^{(r)}_j : j=0,\dots,y_r\,,r=1,\dots,k)$ of random variables (whose dependence on $y$ we choose to omit from the notation for simplicity) as follows. First, set $y_0:=0$ and $W^{(0)}_0:=0$ for convenience and then, for each $r=1,\dots,k$, define recursively $W^{(r)}_0:=W^{(r-1)}_{y_{r-1}}$ and, if $y_r \neq 0$, for each $j=1,\dots,y_r$, 
\[
W^{(r)}_j:=\inf\{ t > W^{(r)}_{j-1} : t \in Y^{(r)}\}.
\] Observe that, by virtue of the independence of $(Y^{(r)})_{r=1,\dots,k}$ and their Markov property, the random variables 
\[
(W^{(r)}_j-W^{(r)}_{j-1} : j=1,\dots,y_r\,,\,r=1,\dots,k \text{ such that }y_r \neq 0)
\] are i.i.d. with Exponential distribution of mean $1$. 

\begin{figure}
 \begin{center}
 \includegraphics[scale=0.65]{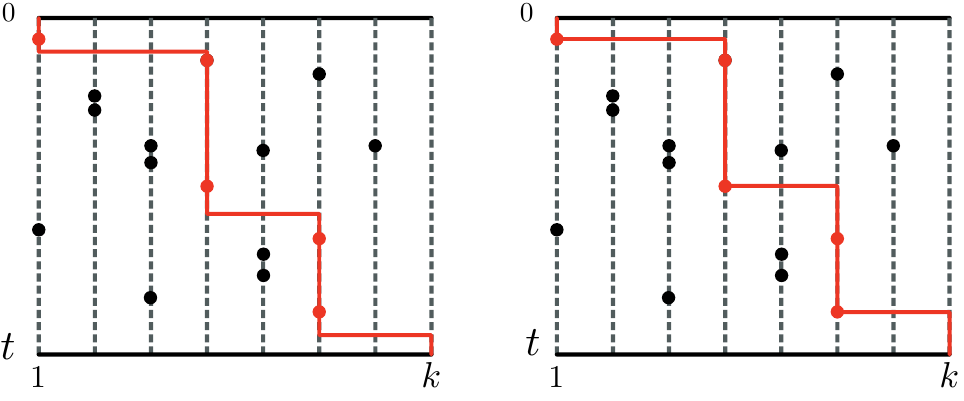}
\caption{\textbf{Illustration of the event $\PLPP(t,k) \geq x$}. The picture on the left shows a path in $\cV(t,k)$ collecting at least $x=5$ discontinuity points of the Poisson processes $Y^{(r)}$, $r=1,\dots,k$. The picture on the right shows how the path on the left can be modified so that all its jumps are at discontinuity points (but it still collects the same amount of them).}\label{fig:3-S}
\end{center} 
\end{figure}


Now, fix any $t > 0$. Note that, if $\PLPP(t,k) \geq x$, there must exist some path $v \in \cV(t,k)$ whose trajectory collects at least $x$ discontinuity points of $(Y^{(r)})_{r=1,\dots,k}$, see Figure~\ref{fig:3-S}. 
Moreover, without loss of generality we may assume that the jumps of $v$ can only occur at discontinuity points of $(Y^{(r)})_{r=1,\dots,k}$ (again, see Figure~\ref{fig:3-S} for an illustration). 
Then, if $t_x \in [0,t]$ denotes the precise time in which $v$ collects the $x$-th point, on the one hand we have
\[
(Y^{(1)}_{v_1 \wedge t_x }-Y^{(1)}_{v_0 \wedge t_x },\dots,Y^{(k)}_{v_k \wedge t_x }-Y^{(k)}_{v_{k-1} \wedge t_x }) \in \cN_{x,k}.
\] On the other hand, since the path $v$ only jumps at discontinuity points of $(Y^{(r)})_{r=1,\dots,k}$, on the event that $(Y^{(1)}_{v_1 \wedge t_x }-Y^{(1)}_{v_0 \wedge t_x },\dots,Y^{(k)}_{v_k \wedge t_x }-Y^{(k)}_{v_{k-1} \wedge t_x })=y$ for some specific $y \in \cN_{x,k}$ we have
\[
v_r \wedge t_x = v_{r-1} \wedge t_x + \sum_{j=1}^{y_r} (W^{(r)}_{j}-W^{(r)}_{j-1})
\] for all $r=1,\dots,k$, with the convention that the sum above is $0$ if $y_r=0$. It follows from these two facts that 
\begin{equation}\label{eq:suma}
\{\PLPP(s,k) \geq x\} = \bigcup_{y \in \cN_{x,k}} \left\{ \sum_{i=1}^k \sum_{j=1}^{y_r} W^{(r)}_j-W^{(r)}_{j-1} \leq t\right\},
\end{equation} which implies the bound
\begin{equation}\label{eq:bound1}
\bP(\PLPP(t,k) \geq x) \leq |\cN_{x,k}| \bP(\textrm{Poisson}(t) \geq x),
\end{equation} since the sum on the event in the right-hand side of \eqref{eq:suma} has $\Gamma(x,1)$ distribution. Now, it is a standard fact that
\begin{equation} \label{eq:combi}
    |\cN_{x,k}| = {x + k -1\choose k-1},
\end{equation}
so that, by combining \eqref{eq:bound1}--\eqref{eq:combi} with Chernoff's inequality for the Poisson distribution,
we arrive at the bound
\begin{equation}\label{eq:bound2}	
\P \lp \PLPP(t,k) \geq x \rp \leq
        \frac{(x+k-1)!}{x!(k-1)!} \rme^{-t} \lp \frac{\rme t}{x} \rp^{x}.
\end{equation}
Finally, by using Stirling's approximation and then performing some standard algebraic manipulations, the right-hand side of~\eqref{eq:bound2} can be further bounded from above by 
\[
\begin{split}C_1 \lp \frac{x+k-1}{k-1} \rp^{k-1} \lp \frac{x+k-1}{x}\rp^x \lp \frac{\rme t}{x} \rp^{x} \rme^{-t} & \leq C_2 \lp 1+\frac{x}{k-1}\rp^{k-1} \lp\frac{(x+k)\rme t}{x^2} \rp^x \rme^{-t}
\\
& \leq C_3 \exp \left[ x \lp \log \lp \frac{(k+x)t}{x^2} \rp +2\rp -t\right]
\end{split}
\] for some absolute constants $C_1,C_2,C_3 > 0$, where, to obtain the last inequality, we used the standard bound $1+\theta \leq \rme^{\theta}$ for all $\theta \in \R$. This gives \eqref{eq:boundD} and concludes the proof.
\end{proof}

\section{Longitudinal fluctuations: proof of Theorem~\ref{theo:1}}\label{sec:proofoftheo1}

In this section, we prove Theorem~\ref{theo:1}. We start by proving Propositions~\ref{prop:1}, \ref{prop:2} and \ref{prop:3} found in the outline of the proof in Section~\ref{sec:outline}, and then use these results to conclude the proof of Theorem~\ref{theo:1}.

\subsection{Proof of Proposition~\ref{prop:1}}
\label{sec:prop1}

Consider a sequence $Y=(Y^{(r)})_{r \in \N}$ of independent Poisson processes $Y^{(r)}=(Y^{(r)}_s)_{s \geq 0}$ of rate $1$ and, for $t > 0$, $k \in \N$ and $G \in \cI$, define 
\[
h^\circ_G(t,k):= \sup_{u \in \cU(t,k)}\left[ H(u,Y) + G(k-u(t)+1)\right].
\] By Lemma~\ref{lemma:1}, we have that $h^\circ_G(t,k)\overset{d}{=}h_G(t,k)$. Furthermore, by definition of $h^\circ_G(t,k)$ it is straightforward to verify that for any pair $G_1,G_2 \in \cI$ such that $G_1 \leq G_2$ we have $h^\circ_{G_1}(t,k) \leq h^\circ_{G_2}(t,k)$, so that, in particular, 
\begin{equation}
\label{eq:mayor}
\sup_{G \in \cI} |h^\circ_G(t,k)-h^\circ_F(t,k)| \leq h^\circ_F(t,k)-h^\circ_S(t,k).
\end{equation}  In light of~\eqref{eq:mayor}, to obtain Proposition~\ref{prop:1} it is enough to show that there exist constants $c_1,c_2,t^* > 0$ such that, for any $t > t^*$ and all $k \in \N$ with $k \log t \leq t$, under this coupling we have that, for all $x > 0$,
\begin{equation}\label{eq:cotag}
\P\Big(h^\circ_F(t,k) - h^\circ_S(t,k) > x\Big) \leq \rme^{c_1 k \log t - c_2 x} + \rme^{-\frac{1}{2}k\log t}.
\end{equation} 

To this end, let us fix $t > \rme$ and $k \in \N$ with $k \log t \leq t$ and define a sequence $(\tau_j)_{j \in \N_0}$ of stopping times as follows. First, we set $\tau_0:= t - k\log t \geq 0$ and then, for each $j \in \N$, we define recursively
\[
\tau_{j}:=\inf\{ s > \tau_{j-1} : s \in Y^{(j)}\}.
\] By the independence of $(Y^{(r)})_{r \in \N}$ and the Markov property, it follows that $W=(W_s)_{s \geq 0}$ given by
\[
W_s:= \sum_{j \in \N} \mathbf{1}_{\{ \tau_j \leq s+\tau_0\}},
\] is a Poisson process of rate $1$. In addition, whenever $W_{k\log t} \geq k-1$ we have $\tau_{k-1} \leq t$, so that there exists at least one path $u \in \cU(t,k)$ such that $u(t)=k$. Indeed, we may~take the path $\widehat{u}$ which, for each $j=1,\dots,k-1$, jumps from $j$ to $j+1$ at time $\tau_j$. In particular, we see that $h_S(t,k) > -\infty$ whenever $W_{k\log t} \geq k-1$. 

Now, fix any $u \in \cU(t,k)$ and, on the event $\{W_{k\log t} \geq k-1\}$,  define a path $u':[0,t] \to \N$ by the formula
\[
u'(s):=\begin{cases} u(s) & \text{ if $s < \tau_0$} \\
								u(\tau_0^-) & \text{ if $s \in [\tau_0,\tau_{u(\tau_0^-)})$}\\
								\widehat{u}(s) & \text{ if $s \in (\tau_{u(\tau_0^-)},t]$}.\end{cases}
\] Put into words, $u'$ is the path that follows $u$ until time $\tau_0$, then stays at its position at time $\tau_0$ until it is intersected by $\widehat{u}$, after which it follows $\widehat{u}$ until time $t$. By construction, it is straightforward to check that $u' \in \cU(t,k)$ and $u'(t)=k$. Moreover, in the notation of \eqref{eq:defhy}, we have
\[
H(u,Y)-H(u',Y)\leq \int_{\tau_0}^t \mathrm{d}Y^{(u(s))}_s.
\] Finally, since the restriction of $u$ to the time interval $[\tau_0,t]$ belongs to the set of paths
\[
\cV([\tau_0,t],k):=\{ v:[\tau_0,t] \to \N \,|\,v \text{ c\`adl\`ag increasing}\,,\,v(t)=k\},
\] we obtain 
\[
H(u,Y)-H(u',Y) \leq \sup_{v \in \cV([\tau_0,t],k)} H(v,Y)=:\PLPP([\tau_0,t],k).
\] Since this bound is independent of $u \in \cU(t,k)$, this implies that 
\[
h^\circ_F(t,k)-h^\circ_S(t,k) \leq \PLPP([\tau_0,t],k),
\] which gives the bound
\begin{equation}\label{eq:descom}
\bP\Big(h^\circ_F(t,k) - h^\circ_S(t,k) > x \Big) \leq \bP(W_{k\log t} < k-1) + \bP\Big( \PLPP([\tau_0,t],k) > x\Big).
\end{equation}
By translation invariance of the Poisson process, we have that $\PLPP([\tau_0,t],k)\overset{d}{=}\PLPP(k\log t,k)$ so that, since $t > \rme$, if $c$ is  sufficiently large (depending only on the value of $C$ in \eqref{eq:boundD}), then, by Lemma~\ref{lemma:3} we have that, for all $x > c k \log t$,
\begin{equation}\label{eq:cotad}
\bP\Big( \PLPP([\tau_0,t],k) > x \Big) \leq \rme^{-x}.
\end{equation}  On the other hand, by Chernoff's bound for the Poisson distribution, 
\begin{equation}\label{eq:cotay}
\bP(W_{k\log t} < k-1) \leq \rme^{-k\log t} \left(\frac{\rme k\log t}{k}\right)^k \leq \rme^{-\frac{1}{2}k\log t}
\end{equation} for all $t$ large enough. Thus, by combining \eqref{eq:cotad} and \eqref{eq:cotay}, in light of \eqref{eq:descom} we immediately obtain the result with $c_1:=\max\{c,1\}$, $c_2:=1$ and $t^*\geq \rme $ large enough so that \eqref{eq:cotad} holds.

\subsection{Proof of Proposition~\ref{prop:2}}
\label{sec:prop2}

Consider a sequence $Y=(Y^{(r)})_{r \in \N}$ of independent Poisson processes of rate $1$ and, for each $t > 0$ and $k \in \N$, define
\[
h^\circ_F(t,k):= \sup_{u \in \cU(t,k)} H(u,Y) 
\] and
\[
\PLPP^\circ(t,k):= \sup_{v \in \cV(t,k)} H(v,Y).
\] By Lemma~\ref{lemma:1}, $(h^\circ_F(t,k),\PLPP^\circ(t,k))$ is a coupling of $h_F(t,k)$ and $\PLPP(t,k)$. Hence, to prove Proposition~\ref{prop:2} it is enough to show that there exist constants $c_1,c_2,t^* > 0$ such that, for any $t>t^*$ and $k \in \N$ with $k \leq t$, for all $x > 0$ we have
\[
\bP( |h^\circ_F(t,k) - \PLPP^\circ(t,k)| > x) \leq \rme^{c_1 \log t - c_2 \frac{x}{k}}.
\] To this end, we first observe that, for any $t > 0$ and $k \in \N$, we have $h^\circ_F(t,k) \leq \PLPP^\circ(t,k)$ since $\cU(t,k) \subseteq \cV(t,k)$ by definition. Hence, it is enough to show that, for each $x > 0$ and $t > 1$, there exists some event $\Omega_{t,k}^x$ with $\P(\Omega_{t,k}^x) \leq \rme^{c_1 \log t - c_2 \frac{x}{k}}$ for all $t$ large enough such that, on the complement of $\Omega^x_{t,k}$, for any $v \in \cV(t,k)$ there exists some $u \in \cU(t,k)$ satisfying
\begin{equation}\label{eq:compara1}
H(v,Y) - H(u,Y) \leq x.
\end{equation}

Thus, if we fix $t > 1$ and $k \in \N$ with $k \leq t$, given $v \in \cV(t,k)$ let us proceed to construct a path $u \in \cU(t,k)$ as in \eqref{eq:compara1}. Notice that if $k=1$ then $\cV(t,k)=\cU(t,k)=\{\overline{u}\}$, where $\overline{u}$ denotes the path constantly equal to $1$ on $[0,t$], so that \eqref{eq:compara1} immediately holds. Hence, for the rest of the proof we assume $k > 1$.

\begin{figure}
 \begin{center}
 \includegraphics[scale=.8]{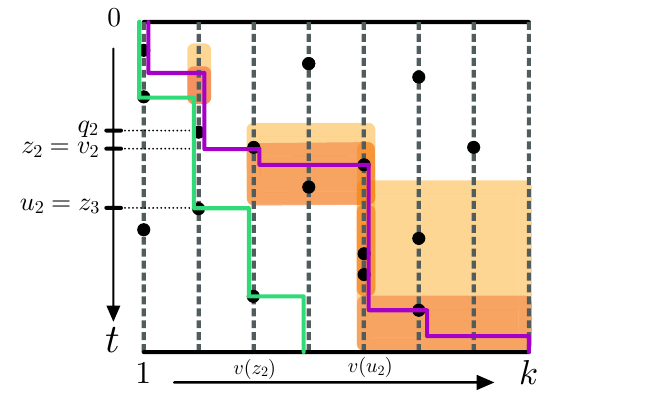}
\caption{\textbf{Construction of the path $u \in \cU(t,k)$ satisfying \eqref{eq:compara1}}. The picture shows, for a given path $v \in \cV(t,k)$ (colored in purple), how to construct a path $u \in \cU(t,k)$ (colored in green) so that \eqref{eq:compara1} holds. The rectangles $[z_r,u_r] \times [v(z_r),v(u_r)]$ are colored in dark orange, while the enlarged rectangles $R_r$ are colored in light orange (although $R_3$ is barely visible because it overlaps with $[z_2,u_2] \times [v(z_2),v(u_r)]$). All ``additional'' points collected by the path $v$ are necessarily contained in the dark orange rectangles.}\label{fig:4-S}
\end{center}
\end{figure}

Recalling that any path $u \in \cU(t,k)$ is uniquely characterized by its vector $(u_0,\dots,u_k)$ (except for the precise value of $u(t)$) , we define $u$ as the unique path in $\cU(t,k)$ which satisfies $u(t)=u(t^-)$, $u_0:=0$ and then, for $r=1,\dots,k$, 
\[
u_r:=\inf \{ s > \max\{u_{r-1},v_r\} : s \in Y^{(r)}\} \wedge t, 
\] see Figure~\ref{fig:4-S} for an illustration. Notice that, if we abbreviate $z_r:=\max\{ u_{r-1},v_r\}$, then by construction the paths $u$ and $v$ disagree precisely during the time intervals $[z_r,u_r)$ with $r=1,\dots,k$ (some of which may be empty if $z_r=u_r$), so that
\[
H(v,Y)-H(u,Y) \leq \sum_{r=1}^k \int_{z_r}^{u_r} \mathrm{d}Y^{(v(s))}_s.
\] Furthermore, observe that during the time interval $[z_r,u_r]$ the graph of $v$ is contained in the rectangle $[z_r,u_r] \times [v(z_r),v(u_r)]$ which is disjoint from $[0,t] \times [1,r]$, see Figure~\ref{fig:4-S}. 
In particular, if we define $q_r:=\sup\{ s \in [0,u_r) : s=0 \text{ or } s \in Y^{(r)}\} \in [0,z_r]$ and consider instead the (time-)enlarged rectangle $
R_r:=[q_r,u_r] \times [v(z_r),v(u_r)]$ (the reason for doing this will be clarified later on), then this last observation above yields the bound
\[
\int_{z_r}^{u_r} \mathrm{d}Y^{(v(s))}_s \leq \sup_{w \in \cV(R_r)} H(w,Y)=:\PLPP(R_r),
\] where, for a given rectangle $R:=[q,q'] \times [l_1,l_2]$ with $q,q' \geq 0$ and $l_1,l_2 \in \N$, we denote 
\[
\cV(R):=\{ w:[q,q'] \to \N \,|\,w \text{ c\`adl\`ag increasing}\,,\,l_1\leq w(s) \leq l_2 \text{ for all }s \in [q,q']\}.
\] Thus, if for $r=1,\dots,k-1$ we enumerate the points in $Y^{(r)} \cup \{0\}$ (where we identify~$Y^{(r)}$ with its set of discontinuity points) as 
$0=:q_0^{(r)} < q_1^{(r)} < \dots,$ then consider the collection of rectangles
\[
\cR^{(r)}_t:=\{ [q^{(r)}_i,q^{(r)}_{i+1} \wedge t] \times [l_1,l_2] : i \in \{0,\dots,Y_t\} \text{ and }r+1 \leq l_1 \leq l_2 \leq k\}
\] 
and finally define the event $\Omega_{t,k}^x$ as
\[
\Omega_{t,k}^x:= \bigcup_{r=1}^{k-1} \bigcup_{R \in \cR^{(r)}_t} \left\{ \PLPP^\circ(R) > \textrm{width}(R)\frac{x}{2k}\right\},
\] where for a rectangle $R:=[q_1,q_2] \times [l_1,l_2]$ we define its width as $\textrm{width}(R):=l_2-l_1+1$,
then on the complement of the event $\Omega_{t,k}^x$ we have that
\[
\sum_{r=1}^k  \int_{z_r}^{u_r} \mathrm{d}Y^{(v(s))}_s \leq \frac{x}{2k} \sum_{r=1}^k (v(u_r)-v(z_r)+1) \leq x, 
\] since $v \in \cV(t,k)$ and $z_r \geq u_{r-1}$ by definition. The definition of $\Omega_{t,k}^x$ does not depend on the particular choice of $v$ (this is why we consider enlarged rectangles). Therefore, to conclude the proof it only remains to show that, for some suitable constants $c_1,c_2>0$, we have $\bP(\Omega_{t,k}^x) \leq \rme^{c_1 \log t - c_2 \frac{x}{k}}$ for all $t$ large enough.

To this end, notice that the union bound gives the estimate
\begin{equation}\label{eq:ub1}
\bP(\Omega_{t,k}^x) \leq \sum_{r=1}^{k-1}  \bP\left( \bigcup_{R \in \cR^{(r)}_t} \left\{\PLPP^\circ(R) > \textrm{width}(R)\frac{x}{2k}\right\}\right).
\end{equation} By conditioning on $Y^{(r)}$, the $r$-th term on the right-hand side of \eqref{eq:ub1} becomes less than
\[
 \sum_{r+1 \leq l_1 \leq l_2 \leq k} \E \left( \sum_{i=0}^{Y^{(r)}_t} \bP\left(\PLPP^\circ([q^{(r)}_i,q^{(r)}_{i+1}] \times [l_1,l_2]) > (l_2-l_1+1))\frac{x}{2k}\,\bigg|\, Y^{(r)} \right) \right),
\] where, to replace $q^{(r)}_{i+1} \wedge t$ by $q^{(r)}_{i+1}$, we used that $\PLPP^\circ([q,q']\times [l_1,l_2])$ is increasing in $q'-q$. Since $(q^{(r)}_{i+1}-q^{(r)}_i)_{i \in \N_0}$ are i.i.d. Exponential random variables with mean $1$ independent of $(Y^{(m)})_{m\geq r+1}$, by translation invariance (in $r$) of the sequence $(Y^{(r)})_{r \in \N}$ we can~bound the last display from above by
\begin{equation}\label{eq:exp}
 \sum_{r+1 \leq l_1 \leq l_2 \leq k} \E \left( (Y^{(r)}_t+1) \bP\left(\PLPP(W, l_2-l_1+1)  > (l_2-l_1+1)\frac{x}{2k}\right) \right),
\end{equation} where $W$ is an Exponential random variable with mean $1$ independent of $(Y^{(r)})_{r \in \N}$.
If, for $\delta \in (0,\frac{1}{2})$, we split into cases depending on whether $W > \delta\frac{x}{k}$ or not, we can bound the probability inside the expectation in \eqref{eq:exp} from above by
\begin{equation}\label{eq:exp2}
\bP\Big(W > \delta \frac{x}{k}\Big) + \bP\Big( \PLPP\Big(\delta\frac{x}{k},l_2-l_1+1\Big) > (l_2-l_1+1)\frac{x}{2k}\Big),
\end{equation} where, once again, we use the fact that $\PLPP(q,l)$ is increasing in $q$. Now, on the one hand we have $\bP(W>\delta \frac{x}{k})= \rme^{-\delta \frac{x}{k}}$ and, on the other hand, by Lemma~\ref{lemma:3} we have that, if $\delta$ is taken small enough (depending only on the value of $C$ in \eqref{eq:boundD}), the rightmost term in \eqref{eq:exp2} is bounded from above by $\rme^{-\frac{x}{2k}}$ for all $x \geq k$. Recalling that $\E(Y^{(r)}_t)=t$, combining these bounds with \eqref{eq:exp} and \eqref{eq:ub1} yields, for all $2 \leq k \leq t$ and $x \geq k$, the estimate
\begin{equation}\label{eq:boundx}
\bP(\Omega_{t,k}^{x}) \leq 2k^3(t+1)\rme^{-\delta \frac{x}{k}} \leq \rme^{6\log t - \delta \frac{x}{k}}.
\end{equation}Since the bound in \eqref{eq:boundx} above holds trivially for all $x \in (0,k]$ whenever $t \geq 2$, in light of \eqref{eq:compara1} we conclude the result with $c_1:=6$, $c_2:=\delta$ and all $t^*:=2$.  

\subsection{Proof of Proposition~\ref{prop:3}}
\label{sec:prop3}

In order to prove Proposition \ref{prop:3}, we must show the existence of constants $c_1,c_2,t^* > 0$ such that, for any $t>t^*$ and $k \in \N$, it is possible to couple $\PLPP(t,k)$ and $\BLPP(t,k)$ so that, for all $x > 0$,
\begin{equation}\label{eq:boundL1}
       \bP(  \left| \PLPP(t,k) - t - \BLPP(t,k)\right| > x) \leq \rme^{c_1k \log t - c_2 x}.
\end{equation}
To construct such coupling, let us fix any $t \geq 1$, $k \in \N$ and for each $r=1,\dots,k$ consider the coupling of a standard Brownian motion $B^{(r)}$ with a Poisson process $Y^{(r)}$ given by Lemma~\ref{lemma:2}. By eventually enlarging the probability space if necessary, we may assume that the processes $(B^{(r)})_{r=1,\dots,k}$ are all defined on the same probability space and independent, and that the same holds for the $(Y^{(r)})_{r=1,\dots,k}$. We now check that $\BLPP(t,k)$ and $\PLPP(t,k)$ defined as in \eqref{eq:defl} and \eqref{eq:defd} respectively using the coupled sequences $(B^{(r)})_{r=1,\dots,k}$ and $(Y^{(r)})_{r=1,\dots,k}$ satisfy \eqref{eq:boundL1}.

To this end, notice that, if we write $Z^{(r)}_s:=Y^{(r)}_s-s$ for each $s \in [0,t]$ and $r=1,\dots,k$, then we have
\begin{align*}
        \PLPP(t,k) &= \sup_{v \in \cV(t,k)} \sum_{r = 1}^k  Y_{v_r}^{(r)} - Y_{v_{r-1}}^{(r)}
        \\
        &=
        \sup_{v \in \cV (t,k)} \sum_{r = 1}^k  Y_{v_r}^{(r)} - v_r - \lp Y_{v_{r-1}}^{(r)} - v_{r-1} \rp + \sum_{r = 1}^k v_r - v_{r-1}
        \\
        &=
        \sup_{v \in \cV (t,k)} \sum_{r = 1}^k \left[Z_{v_r} - Z_{v_{r-1}} \right]+ t
        \\
        &=:
        J(t,k) + t.
    \end{align*}
    Hence, to show the bound, it is enough to prove that there exist $c_1,c_2,t^*> 0$ such that
    \begin{equation}\label{eq:boundG}
        \bP ( \left| J(t,k) - \BLPP(t,k )\right| > x )  \leq \rme^{c_1 k \log t - c_2 x}
   \end{equation} holds for all $x > 0$ and $t> t^*$. 
    To this end, notice that, as in \cite[Theorem~1]{Bodineau}, 
    \begin{align*} 
        |J(t,k) - \BLPP(t,k)| &= \left| \sup_{v \in \cV (t,k)} \sum_{r = 1}^k \left[Z_{v_r}^{(r)} - Z_{v_{r-1}}^{(r)} \right] - \sup_{v' \in \cV (t,k)} \sum_{r = 1}^k \left[B_{v'_r}^{(r)} - B_{u'_{r-1}}^{(r)} \right]\right|
        \\
        &\leq
        \sup_{v \in \cV(t,k)} \sum_{r = 1}^k \left| Z_{v_r}^{(r)} -B_{u_r}^{(r)} \right| + \left| Z_{u_{r-1}}^{(r)} - B_{u_{r-1}}^{(r)}\right|
        \leq
        2 \sum_{r =1}^k \sup_{s \in [0,t]} \left| Z_s^{(r)} - B_s^{(r)} \right|.
    \end{align*} Hence, by the exponential Tchebychev inequality, for any $\theta \geq 0$ we have that
\begin{align}\label{eq:boundG2}
\bP( | J(t,k) - \BLPP(t,k)| > x ) & \leq \rme^{-\frac{\theta}{2} x} \E\Big(\exp\Big\{ \theta \sum_{r =1}^k \sup_{s \in [0,t]} \left| Z_s^{(r)} - B_s^{(r)} \right|\Big\}\Big) \nonumber \\ &= \rme^{-\frac{\theta}{2}x}\left( \E\left( \exp\Bigg\{\theta \sup_{s \in [0,t]} \left| Z_s^{(1)} - B_s^{(1)} \right|\Bigg\}\right)\right)^k
\end{align} since the processes $(Z^{(r)}-B^{(r)})_{r=1,\dots,k}$ are i.i.d. 
Recalling the constants $C,\lambda$ and $\kappa$ from Lemma~\ref{lemma:2}, if we take $\theta:=\frac{\kappa}{2}$, then by this very lemma we obtain
\begin{align*}
\E\left( \exp\Bigg\{\theta \sup_{s \in [0,t]} \left| Z_s^{(1)} - B_s^{(1)} \right|\Bigg\}\right) &= \int_0^\infty \bP \Bigg(  \sup_{s \in [0,t]} \left| Z_s^{(1)} - B_s^{(1)} \right| > \frac{1}{\theta}\log s\Bigg) \mathrm{d}s\\ & \leq t^{C\theta} + Kt^\lambda \int_{t^{C\theta}}^\infty \rme^{-\frac{\kappa}{\theta} \log s}\mathrm{d}s\\ 
& = t^{C\kappa/2} + Kt^{\lambda - C\kappa/2} \leq (K+1)\rme^{c \log t}
\end{align*} for $c:=\max\{C\kappa/2, \lambda - C\kappa/2\} > 0$. In combination with \eqref{eq:boundG2}, this immediately gives~\eqref{eq:boundG} for $c_1:=c+1$, $c_2:=\frac{\kappa}{4}$ and $t^*:= \max\{ \log (K+1) ,1\}$, and thus concludes the proof. 

\subsection{Conclusion of the proof of Theorem~\ref{theo:1}}\label{sec:conclusion}

By the discussion from Section~\ref{sec:outline}, to conclude the proof of Theorem~\ref{theo:1} we only need to show that there exist $c_1,c_2,t^*> 0$ such that, given any $t > t^*$ and $k \leq \frac{t}{\log t}$, there exists a joint coupling of the variables $(h_G(t,k) : G \in \cI)$, $\PLPP(t,k)$ and $\BLPP(t,k)$ in such a way that \eqref{eq:c1}, \eqref{eq:c2} and \eqref{eq:c3} all hold.

To do this, we fix any $t \geq 1$, $k \in \N$ and for each $r=1,\dots,k$ we consider the coupling given by Lemma~\ref{lemma:2} of a standard Brownian motion $B^{(r)}$ with a Poisson process $Y^{(r)}$ of rate $1$. As discussed in the proof of Proposition~\ref{prop:3}, we may assume that the processes $(B^{(r)})_{r=1,\dots,k}$ are all defined on the same probability space and independent, and that the same holds for the $(Y^{(r)})_{r=1,\dots,k}$. We then define
\[
\widehat{\BLPP}(t,k):= \sup_{v \in \cV(t,k)} H(v,B^{(1)},\dots,B^{(k)}),
\]
\[
\widehat{\PLPP}(t,k):= \sup_{v \in \cV(t,k)} H(v,Y^{(1)},\dots,Y^{(k)}),
\] and, for $G \in \cI$,
\[
\widehat{h}_G(t,k):= \sup_{u \in \cU(t,k)} [H(u,Y^{(1)},\dots,Y^{(k)}) + G(k-u(t)+1)],
\] where $\cV(t,k)$ and $H$ are respectively given by \eqref{eq:defvtk} and \eqref{eq:defh}--\eqref{eq:defhy}, and $\cU(t,k)$ is defined as in \eqref{eq:defu} but using the coupled sequence $(Y^{(r)})_{r=1,\dots,k}$. It follows from the proofs of Propositions \ref{prop:1}--\ref{prop:2}--\ref{prop:3} that \eqref{eq:c1},\eqref{eq:c2} and \eqref{eq:c3} all hold if $t$ is taken sufficiently large, for some appropriate choice of constants $c_1,c_2 > 0$. From here, Theorem~\ref{theo:1} follows simply from the triangular inequality and the union bound.

\section{Moderate deviations: proof of Theorem~\ref{theo:3}}\label{sec:proofoftheo3}

By Theorem~\ref{theo:2} there exists some $t^* > 0$ such that, for any $t > t^*$ and $k \leq \frac{t}{\log t}$, one can couple $h_G(t,k)$ and $\BLPP(t,k)$ so that, for any $\varepsilon,x>0$,
\[
\bP( |h_G(t,k) - t - \BLPP(t,k)| > \varepsilon x \sqrt{t}k^{-1/6}) \leq \exp\{ c_1 \log t - c_2 \varepsilon x \sqrt{t} k^{-7/6}\} + \rme^{-\frac{1}{2}k \log t},
\] for some absolute constants $c_1,c_2 > 0$. 

Now, on the one hand, for any $x \leq (\frac{1}{4}k\log t)^{2/3}$ we immediately have the bound
\begin{equation}\label{eq:cotab1}
\rme^{-\frac{1}{2}k\log t} \leq \rme^{-2x^3}\mathbf{1}_{x \leq (\frac{1}{4}k\log t)^{1/3}} + \rme^{-2x^{\frac{3}{2}}}\mathbf{1}_{x > (\frac{1}{4}k\log t)^{1/3}}.
\end{equation} On the other hand, if $x \geq 1$, then, since $\lim_{t \to \infty} \frac{\alpha(t)}{t^{\frac{3}{7}}(\log t)^{-\frac{6}{7}}}=0$, for all $t$ sufficiently large (depending only on $\alpha$, $\varepsilon$, $c_1$ and~$c_2$) and any $k \in [1,\alpha(t)]$, we have
\[
c_1 \log t - c_2 \varepsilon x \sqrt{t}k^{-7/6} =\log t \left( c_1 - c_2 \varepsilon x \frac{\sqrt{t}}{k^{7/6}\log t}\right) \leq -2x\log t
\] so that, if in addition $x \leq (\log t)^2$, we obtain
\begin{equation}
\label{eq:cotab2}
\exp\{ c_1 \log t - c_2 \varepsilon x \sqrt{t} k^{-7/6}\} \leq \rme^{-2x^3}\mathbf{1}_{x \leq \sqrt{\log t}}+ \rme^{-2x^{3/2}}\mathbf{1}_{x > \sqrt{\log t}}.
\end{equation} Thus, if we abbreviate $b_{t,k}:=(\frac{1}{4}k\log t)^{1/3} \wedge \sqrt{\log t}$, then by \eqref{eq:cotab1}--\eqref{eq:cotab2} we conclude that, if $t$ is sufficiently large, then, for all $k \in [1,\alpha(t)]$ and $x \in [1, (\frac{1}{4} k \log t)^{2/3} \wedge (\log t)^2]$,
\begin{equation}\label{eq:cotab3}
\bP( |h_G(t,k) - t - \BLPP(t,k)| > \varepsilon x \sqrt{t}k^{-1/6}) \leq 2\rme^{-2x^3}\mathbf{1}_{x \leq b_{t,k}} + 2\rme^{-2x^{3/2}}\mathbf{1}_{x > b_{t,k}}.
\end{equation}
Recalling the moderate deviation estimates for $\BLPP(1,k)$ derived in \cite{Basu} (in particular, see Propositions 1.5 and 1.8 in \cite{Basu}) and the Brownian scaling relation $\BLPP(t,k)\overset{d}{=}\sqrt{t}\BLPP(1,k)$, the estimates \eqref{eq:est1}--\eqref{eq:est2} now follow at once from \eqref{eq:cotab3} by a standard computation involving the union bound.

Finally, to see \eqref{eq:est3}, we take $\varepsilon_t:=(\log t)^{-3}$ and notice that, 
if $\lim_{t \to \infty} \frac{\alpha(t)}{t^{\frac{3}{7}}(\log t)^{-\frac{24}{7}}}=0$, by repeating the same computations as above with $\varepsilon_t$ in place of $\varepsilon$ we obtain that, if $t$ is sufficiently large, then, for all $k \in [1,\alpha(t)]$ and $x \in [1, (\frac{1}{4} k \log t)^{2/3} \wedge (\log t)^2]$,
\begin{equation}\label{eq:cotab4}
\bP( |h_G(t,k) - t - \BLPP(t,k)| > \varepsilon_t x \sqrt{t}k^{-1/6}) \leq 2\rme^{-2x^{3/2}}.
\end{equation} Since $|\varepsilon_t x^{3/2}| \leq 1$ for all $x \in [1,(\log t)^2]$, \eqref{eq:est3} follows from the estimate in \cite[Lemma~7.3]{PaquetteZeitouni} and \eqref{eq:cotab4} by a computation analogous to the one yielding \eqref{eq:est1}--\eqref{eq:est2}. We omit the details.

\section{Preliminaries for the proof of Theorem~\ref{theo:4}}\label{sec:prel2}
\label{preliminaries.trans.fluct}

In this section we prove all preliminary results needed to establish Propositions~\ref{prop:fluc1} and \ref{prop:fluc2}. More precisely, we~give the proofs of Lemmas~\ref{lemma:fluc1a}, \ref{lemma:fluc1b} and \ref{lemma_geodesics_2-0}.

\subsection{Proof of Lemma~\ref{lemma:fluc1a}} We begin with the bounds for $\max_{k \in C^\gamma(t,\alpha(t);s)} U_{t,\alpha}(s,k)$. To this end, we first observe that the upper bound in \eqref{eq:maxleft} and the lower bound in \eqref{eq:maxright} are direct consequences of Theorem~\ref{theo:3}, so we only need to show the other two bounds. The proof of the upper bound in \eqref{eq:maxright} is analogous to that of \cite[Proposition~1.9]{Basu}, so we will only prove the lower bound in \eqref{eq:maxleft}, which follows by a similar argument.

We begin by noticing that, given $\varepsilon > 0$, if we set $U_{t,\alpha}^\gamma(s):=\max_{k \in C^\gamma(t,\alpha(t);s)}U_{t,\alpha}(s,k)$, then, for all $t$ large enough, we have
\[
\P\left( U^\gamma_{t,\alpha}(s) \leq -x \right) \geq \P \left( \max_{k \in C^\gamma(t,\alpha(t);s)} h^\circ_S(t,k) - st - 2\sqrt{stk} \leq -\Big(1+\frac{\varepsilon}{2}\Big){\sqrt{st(s\alpha(t))^{-1/3}}}x\right).
\] Now, define $\delta:=\frac{2}{3}-\gamma > 0$ and {let us consider the space-time point} 
\[
\vec{v}:=(v_\text{ti},v_\text{sp}):= ( st + t(\alpha(t))^{-\delta} , s\alpha(t) + (\alpha(t))^{1-\delta}),
\] where we write $v_{\text{ti}}$ and $v_{\text{sp}}$ to respectively denote the time and space coordinates of $v$. Observe on the one hand that, for all $j \in C^\gamma(t,\alpha(t);s)$, 
\begin{equation}\label{eq:desig}
h^\circ_S(v_{\text{ti}},v_{\text{sp}}) \geq h^\circ_S(st,j)+ h_S^\circ((st,j) \to \vec{v}),
\end{equation} with $h_S^\circ((st,j) \to \vec{v})$ defined as in \eqref{eq:defhpp} (but for $G=S$). On the other hand, since $\gamma < \frac{2}{3}$, it is straightforward to check by first order Taylor expansion that, for any $j \in C^\gamma(t,\alpha(t);s)$, as $t$ tends to infinity, 
\[
2\sqrt{stj} =2 \sqrt{(st)(s\alpha(t))} + x_j \sqrt{t}((\alpha(t))^{\gamma-\frac{1}{2}}+O^{(1)}_s(n_s(t))
\] and
\[
2\sqrt{(v_{\text{ti}}-st)(v_{\text{sp}}-j)} =2\sqrt{(v_{\text{ti}}-st)(v_{\text{sp}}-s\alpha(t))} -x_j\sqrt{t}(\alpha(t))^{\gamma-\frac{1}{2}}+O^{(2)}_s(n_s(t))
\] where we abbreviate $x_j:=\frac{j-s\alpha(t)}{(\alpha(t))^\gamma} \in [-1,1]$ and $n_s(t):=\sqrt{st(s\alpha(t))^{-1/3}}$. In particular, since $\sqrt{v_{\text{ti}}v_{\text{sp}}}=\sqrt{(st)(s\alpha(t))}+\sqrt{(v_{\text{ti}}-st)(v_{\text{sp}}-s\alpha(t))}$ by definition of~$\vec{v}$, we obtain~that 
\[
2\sqrt{stj}+2\sqrt{(v_{\text{ti}}-st)(v_{\text{sp}}-j)} = 2\sqrt{v_{\text{ti}}v_{\text{sp}}}+O^{(3)}_s(n_s(t)).
\] Thus, if we define the events
\[
\mathcal{A}:=\left\{ h^\circ_S(v_{\text{ti}},v_{\text{sp}})-v_{\text{ti}} - 2\sqrt{v_{\text{ti}}v_{\text{sp}}} \leq -(1+\varepsilon)n_s(t)x\right\},
\]
\[
\mathcal{B}:=\left\{ \min_{j \in C^\gamma(t,\alpha(t);s)} \left( h_S^\circ((st,j) \to \vec{v}) -(v_{\text{ti}}-st)-2\sqrt{(v_{\text{ti}}-st)(v_{\text{sp}}-j)}\right) \leq -\frac{\varepsilon}{4}n_s(t)x\right\}
\] and
\[
\mathcal{C}:=\left\{ \max_{j \in C^\gamma(t,\alpha(t);s)} (h^\circ_S(t,j) - st - 2\sqrt{stj}) \leq -\Big(1+\frac{\varepsilon}{2}\Big)n_s(t)x\right\},
\] then, for $x$ and $t$ taken sufficiently large depending only on $s,\gamma$ and $\varepsilon$, we have $\mathcal{A} \subseteq \mathcal{B} \cup \mathcal{C}$, so that
\[ 
\P (\mathcal{C}) \geq \P(\mathcal{A}) - \P(\mathcal{B}).
\] Now, since $\sqrt{v_{\text{ti}}v_{\text{sp}}^{-1/3}}=(1+o_{s.\gamma}(1))n_s(t)$ as $t \to \infty$, by Theorem~\ref{theo:3} we see~that, for all $t,x$ large enough (depending only on $\gamma$, $s$ and $\varepsilon$) such that $x \leq \min\{ (\alpha(t))^{1/10},\sqrt{\log st}\}$,
\[
\P(\mathcal{A}) \geq \mathrm{e}^{-\frac{1}{12}(1+\varepsilon)^5 x^3}.
\] On the other hand, since 
\begin{equation}\label{eq:cotapa}
\frac{n_s(t)}{\sqrt{(v_{\text{ti}}-st)(v_{\text{sp}}-j)^{-1/3}}}=(1+o(1))s(\alpha(t))^{\frac{\delta}{3}}
\end{equation} holds as $t \to \infty$ uniformly over $j \in C^\gamma(t,\alpha(t);s)$ and $h^\circ_G((st,j) \to \vec{v})\overset{d}{=}h_G^\circ(v_{\text{ti}}-st,v_{\text{sp}}-j)$ by shift invariance of the sequence $(Y^{(r)})_{r \in \N}$ and translation invariance of each $Y^{(r)}$, Theorem~\ref{theo:3} together with the union bound yields that, for all $x \geq \frac{4}{\varepsilon}$ and $t$ large enough (depending only on $\gamma$, $s$ and $\varepsilon$), 
\begin{equation*}
\begin{split}
\P(\mathcal{B}) &\leq (\alpha(t))^{\gamma}\left( \rme^{-\frac{1}{12}(1-\varepsilon) (\frac{\varepsilon s}{4})^3  (\alpha(t))^{\delta}x^3}+\rme^{-\frac{1}{12}(1-\varepsilon) (\alpha(t))^{3/10}}+\rme^{-\frac{1}{12}(1-\varepsilon)(\log t)^{3/2}}\right) \\ & \leq \rme^{-(\alpha(t))^{\delta/2}x^3}+\rme^{-\frac{1}{24}(\alpha(t))^{3/10}}+\rme^{-\frac{1}{24}(\log t)^{3/2}}
\end{split}
\end{equation*} Upon noticing that $x^6 \leq \min\{ (\alpha(t))^{3/10},(\log t)^{3/2}\}$ when $x \leq \min\{(\alpha(t))^{1/20},(\log t)^{1/4}\}$, we conclude that, for all $t,x$ are large enough (depending only on $\gamma$, $s$ and $\varepsilon$) satisfying $x \leq \min\{(\alpha(t))^{1/20},(\log t)^{1/4}\}$, we have
\[
\P(\mathcal{B}) \leq \rme^{-(\alpha(t))^{\delta/2}x^3}+2\rme^{-\frac{1}{24}x^6} \leq \frac{1}{2}\rme^{-\frac{1}{12}(1+\varepsilon)^5 x^3},
\] which, in light of \eqref{eq:cotapa}, implies the desired result since $\varepsilon$ can be chosen arbitrarily small.

The proof of the tail estimates for $\max_{k \in C^\gamma(t,\alpha(t);s)} V_{t,\alpha}(s,k)$ is essentially analogous. The only difference is that, in place of $\vec{v}$, we must now consider {the space-time point}
\[
\vec{w}:=(w_{\text{ti}},w_{\text{sp}})=(st-t(\alpha(t))^{-\delta},s\alpha(t) - (\alpha(t))^{1-\delta})
\] and, instead of \eqref{eq:desig}, use that, for all $j \in C^\gamma(t,\alpha(t);s)$,
\[
h^\circ_F(w_{\text{ti}},w_{\text{sp}}) \geq h^\circ_S(\vec{w} \to (st,j))+h^\circ_F((st,j) \to (t,\alpha(t))).
\] From here, the proof continues as in the previous case, we omit the details. 

\subsection{Proof of Lemma~\ref{lemma:fluc1b}} For simplicity, let us write $a_s:=s^{-\frac{1}{3}}$ and $b_s:=(1-s)^{-\frac{1}{3}}$. Now, let us fix $\varepsilon > 0$, $\tilde{x} \geq x^*_\varepsilon$ and write $z_{\varepsilon}^*:=z_{\varepsilon,s}(\tilde{x})$ for simplicity. Then, given $p \in (0,1)$, if $\varepsilon > 0$ is small enough so as to have $p<1-\varepsilon$, we have 
\begin{equation}\label{eq:cotaint}
\P\left( \frac{1}{a_s}U_{t} + \frac{1}{b_s} V_t  \leq - z_{\varepsilon}^* \right) \geq \int_{-(1-\varepsilon)b_s z_{\varepsilon}^*}^{-pb_sz_{\varepsilon,s}} \P\left(U_{t}   \leq - a_sz_{\varepsilon}^*- \frac{a_s}{b_s} y\right) \P( V_t \in \mathrm{d}y).
\end{equation} Since $a_sz_\varepsilon^* + \frac{a_s}{b_s}y \in [x^*_\varepsilon, (1-p)a_sz_\varepsilon^*]$ for all $y \in [-(1-\varepsilon)b_sz_\varepsilon^*,-pb_sz_\varepsilon^*]$, if $t$ is large enough so that $a_s z^*_\varepsilon \leq \beta(t)$ and $t > t^*_\varepsilon$, we can bound the right-hand side of \eqref{eq:cotaint} from below by
\begin{equation}\label{eq:cotaint2}
\rme^{-\frac{1}{12}(1+\varepsilon)(1-p)^3a_s^3(z_\varepsilon^*)^3} \P( -(1-\varepsilon)b_sz_\varepsilon^* \leq V_t \leq - pb_sz_\varepsilon^*).
\end{equation} Moreover, if in addition we take $\varepsilon$ small enough so that $\varepsilon < p$ and also $t$ large enough so that $b_sz_\varepsilon^* \leq \beta(t)$ and $t > t^*_\varepsilon$, then the inclusion $[pb_sz_\varepsilon^*,(1-\varepsilon)b_sz_\varepsilon^*]\subseteq [x_\varepsilon^*,\beta(t)]$ holds, so that \eqref{eq:cotaint2} is bounded from below by
\[
\rme^{-\frac{1}{12}(1+\varepsilon)(1-p)^3a_s^3(z_\varepsilon^*)^3} (\rme^{-\frac{1}{12}(1+\varepsilon)p^3b_s^3(z_\varepsilon^*)^3}-\rme^{-\frac{1}{12}(1-\varepsilon)^4b_s^3(z_\varepsilon^*)^3})
\] Now, if $\varepsilon$ is chosen small enough (depending only on $p$ and $s$) so that $(1+\varepsilon)p^3<(1-\varepsilon)^4$ and $z_\varepsilon^* \gg 1$ (here we use that $\tilde{x}\geq x_\varepsilon^* \geq 1$ so that $z_\varepsilon^* \to \infty$ as $\varepsilon \to 0$ uniformly over $\tilde{x}$), then 
\[
\frac{1}{2}\rme^{-\frac{1}{12}(1+\varepsilon)p^3b_s^3(z_\varepsilon^*)^3} \geq \rme^{-\frac{1}{12}(1-\varepsilon)^4b_s^3(z_\varepsilon^*)^3}
\] which, in combination with all the previous estimates, yields that, for any fixed $p \in (0,1)$, if $\varepsilon$ is taken small enough (depending only on $p$ and $s$), then
\[
\P\left( \frac{1}{a_s}U_{t} + \frac{1}{b_s} V_t  \leq - z_\varepsilon^* \right) \geq \frac{1}{2}\exp\bigg\{-\frac{1}{12}(1+\varepsilon)[(1-p)^3a_s^3+p^3b_s^3](z_\varepsilon^*)^3\bigg\}
\] for all $t$ sufficiently large (depending only on $\varepsilon$, $s$ and $\tilde{x}$). Finally, choosing 
\[
p:=\frac{1}{1+\sqrt{\frac{s}{1-s}}}
\] gives the lower bound in \eqref{eq:cotaint3}.

\subsection{Proof of Lemma~\ref{lemma_geodesics_2-0}} 
Write $k_{t,\alpha}^\gamma(s)=\frac{s}{t}\alpha(t)+ (\alpha(t))^\gamma - \Delta$, for some $\Delta \in [0,1]$. Then, for all $t$ large enough,
\[
\sqrt{sk_{t,\alpha}^\gamma(s)} + \sqrt{(t-s)(\alpha(t)-k_{t,\alpha}^\gamma(s))} - \sqrt{t\alpha(t)} = \sqrt{t\alpha(t)} F(u,v),
\] where $v:=(\alpha(t))^{\gamma-1}-\frac{\Delta}{\alpha(t)} \in (0,1)$, $u:=\frac{s}{t} \in [0,1-v]$ and  
\[
F(u,v)=\sqrt{u(u+v)} + \sqrt{(1-u)(1-u-v)}-1.
\] 

Now, for any fixed $y \in (0,\frac{3}{4})$, by Taylor expansion we have that, for $x \in [\frac{1}{3} y,1-y]$, 
\[
F(x,y)=x\sqrt{1+\tfrac{y}{x}} + (1-x)\sqrt{1-\tfrac{y}{1-x}}-1 \leq -\frac{1}{8}\left(\frac{1}{(1+\frac{y}{x})^{\frac{3}{2}}}\frac{y^2}{x}+\frac{y^2}{1-x}\right) \leq -\frac{1}{16}y^2
\] while, for $x \in [0,\frac{1}{3} y]$, 
\[
F(x,y) \leq \frac{2}{3}y + \sqrt{1-y}-1 \leq - \frac{1}{6}y,
\] which together yield that, for all $t$ sufficiently large (depending only on $\alpha$ and $\gamma$) so as to have $v \in (0,\tfrac{3}{4})$, 
\[
\max_{u \in [0,1-v]}F(u,v) \leq -\frac{1}{16}v^2.
\] From here, the result now follows by taking $t$ large enough so that in addition we have $v^2 \geq \frac{1}{2}(\alpha(t))^{2(\gamma-1)}$.

\section{Transversal fluctuations: proof of Theorem~\ref{theo:4}} \label{sec:proofoftheo4}

We now give the proof of Theorem~\ref{theo:4}. Recall that, by the discussion in the outline of the proof in Section~\ref{sec:outline2}, it is enough to prove Propositions~\ref{prop:fluc1} and \ref{prop:fluc2}. We do this below.

\subsection{Proof of Proposition~\ref{prop:fluc1}}\label{sec:proplbf}

For $t,x>0$, let us consider the event 
\[
\Omega_G(t,\alpha,x):=\left\{\frac{h^\circ_G(t,\alpha(t)) - t - 2\sqrt{t\alpha(t)}}{\sqrt{t(\alpha(t))^{-\frac{1}{3}}}} \leq -x\right\}.
\] Then, by the union bound we have, for any $\gamma> 0$ and $s \in (0,1)$,
\begin{equation}\label{eq:ubound}
\P(B_G^\gamma(t,\alpha(t);s))\leq \P(\Omega_G(t,\alpha,x)) + \P(B_G^\gamma(t,\alpha(t);s)\setminus \Omega_G(t,\alpha,x)).
\end{equation} Now, if $\overline{u} \in \Pi(t,\alpha(t))$ is any geodesic for $h^\circ_G(t,\alpha(t))$ and we abbreviate $\overline{u}_*:=\overline{u}(st)$, then by definition of geodesic we obtain that
\begin{equation}\label{eq:decomph}
h^\circ_G(t,\alpha(t)) = h_S^\circ((0,1) \to (st,\overline{u}_*)) + h^\circ_G( (st,\overline{u}_*) \to (t,\alpha(t))),
\end{equation}
where, for $t_1\leq t_2$ and $k_1 \leq k_2$, the quantity $h^\circ_G((t_1,k_1) \to (t_2,k_2))$ is defined as in \eqref{eq:defhpp}, see Figure~\ref{fig:C-2} for an illustration.
Moreover, since $G \leq F$, we have the bound 
\[
h^\circ_G( (st,\overline{u}_*) \to (t,\alpha(t))) \leq h^\circ_F( (st,\overline{u}_*) \to (t,\alpha(t))).
\]
Therefore, using \eqref{eq:decomph}, if we set $p_s:=\frac{\overline{u}_*}{\alpha(t)}$ and $q_s:=1-p_s$, then we may bound
\[
\frac{h^\circ_G(t,\alpha(t)) - t - 2\sqrt{t\alpha(t)}}{\sqrt{t(\alpha(t))^{-\frac{1}{3}}}} \leq \sqrt{ sp_s^{-\frac{1}{3}}}U_{t,\alpha}(s,\overline{u}_*) + \sqrt{ (1-s)q_s^{-\frac{1}{3}}}V_{t,\alpha}(s,\overline{u}_*) + 2R(t),
\] where
\[
U_{t,\alpha}(s,\overline{u}_*):=\frac{h^\circ_S((0,1) \to (st,\overline{u}_*)) - st - 2\sqrt{st\overline{u}_*}}{\sqrt{st(\overline{u}_*)^{-\frac{1}{3}}}}, 
\]
\[
V_{t,\alpha}(s,\overline{u}_*):=\frac{h^\circ_F( (st,\overline{u}_*) \to (t,\alpha(t))) - (1-s)t - 2\sqrt{(1-s)t(\alpha(t)-\overline{u}_*)}}{\sqrt{(1-s)t(\alpha(t)-\overline{u}_*))^{-\frac{1}{3}}}}
\] and
\[
R(t):= \frac{\sqrt{stp_s\alpha(t)} + \sqrt{(1-s)tq_s\alpha(t)} - \sqrt{t\alpha(t)}}{\sqrt{t(\alpha(t))^{-\frac{1}{3}}}}.
\] Notice that, by a straightforward calculation, we have
\[
R(t)= (\alpha(t))^{\frac{2}{3}}(\sqrt{sp_s}+\sqrt{(1-s)q_s}-1) \leq 0.
\] Hence, on the event $B_G^\gamma(t,\alpha(t);s)$ there exists at least one geodesic $\overline{u}$ for which we have $|\overline{u}_*-\alpha(t)s|\leq ((\alpha(t))^\gamma$, which implies that 
\begin{equation}\label{eq:ineqevento}
\frac{h^\circ_G(t,\alpha(t)) - t - 2\sqrt{t\alpha(t)}}{\sqrt{t(\alpha(t))^{-\frac{1}{3}}}} \leq \sqrt{sp_s^{-\frac{1}{3}}}U_{t,\alpha}^\gamma(s) + \sqrt{(1-s)q_s^{-\frac{1}{3}}}V_{t,\alpha}^\gamma(s),
\end{equation} where we define
\[
U_{t,\alpha}^\gamma(s):=\max_{k \in C^\gamma(\alpha(t);s)} U_{t,\alpha}(s,k)
\] and 
\[
V_{t,\alpha}^\gamma(s):=\max_{k \in C^\gamma(t,\alpha(t);s)} V_{t,\alpha}(s,k), 
\] with $C^\gamma(t,\alpha(t);s):=\{ k \in \N : |k -s\alpha(t)| \leq (\alpha(t))^\gamma\}$ being the cross-section of $C^\gamma(t,\alpha(t))$ at time $t'=st$. In addition, on the event $B^\gamma_G(t,\alpha(t);s)$ we also have (for $p_s$ and $q_s$ defined with respect to this particular geodesic $\overline{u}$)
\[
p_s:=s + \frac{\overline{u}_*-s\alpha(t)}{\alpha(t)} = s + O_1((\alpha(t))^{\gamma-1})
\]  and, similarly, $q_s:=1-s + O_2((\alpha(t))^{\gamma-1})$, so that
\begin{equation}\label{eq:ineqpesos}
\sqrt{sp_s^{-\frac{1}{3}}}=s^{\frac{1}{3}}+O_{1,s}((\alpha(t))^{\gamma-1}) \quad\text{and}\quad\sqrt{(1-s)q_s^{-\frac{1}{3}}}=(1-s)^{\frac{1}{3}}+O_{2,s}((\alpha(t))^{\gamma-1}),
\end{equation} where the $O_{i,s}((\alpha(t))^{\gamma-1})$ terms satisfy that there exist some $t_{\alpha,s},C_{\alpha,s} > 0$ such that, for all $t > t_{\alpha,s}$, 
\[
\max_{i=1,2}|O_{i,s}((\alpha(t))^{\gamma-1})| \leq C_{\alpha,s}(\alpha(t))^{\gamma-1}.
\]
In particular, given any $\varepsilon \in (0,s^{1/3}+(1-s)^{1/3})$, if we fix some $x'_\varepsilon$ (to be specified later) satisfying $x'_\varepsilon \geq x_\varepsilon$, with $x_\varepsilon$ as in Theorem~\ref{theo:3}, and set 
\[
z_{\varepsilon,s}:=(s^{\frac{1}{3}}+(1-s)^{\frac{1}{3}})\tfrac{1}{\varepsilon}x'_\varepsilon >  x_\varepsilon,
\] then, on the one hand, if $t$ is sufficiently large so as to have $z_{\varepsilon,s} \leq \min\{ (\alpha(t))^{1/10},\sqrt{\log t}\}$, by Theorem~\ref{theo:3} we have
\[
\P(\Omega_G(t,\alpha,z_{\varepsilon,s})) \leq \rme^{-\frac{1}{12}(1-\varepsilon)z_{\varepsilon,s}^3}.
\] On the other hand, by \eqref{eq:ineqevento}--\eqref{eq:ineqpesos} and the union bound we obtain 
\begin{equation}
\P(B_G^\gamma(t,\alpha(t);s)\setminus \Omega_G(t,\alpha,x)) \leq \P\left(s^{\frac{1}{3}}U_{t,\alpha}^\gamma(s) + (1-s)^{\frac{1}{3}}V_{t,\alpha}^\gamma(s) > -(1+\varepsilon)z_{\varepsilon,s}\right) + \textrm{E}_{t,\alpha}(s),
\end{equation}
with 
\[
\textrm{E}_{t,\alpha}(s):=\P\left( |U^\gamma_{t,\alpha}(s)| \geq \frac{\varepsilon}{2}\frac{z_{\varepsilon,s}}{C_{\alpha,s}}(\alpha(t))^{1-\gamma}\right) + \P\left( |V^\gamma_{t,\alpha}(s)| \geq \frac{\varepsilon}{2}\frac{z_{\varepsilon,s}}{C_{\alpha,s}}(\alpha(t))^{1-\gamma}\right).
\] 
Therefore, in light of \eqref{eq:ubound}, in order to obtain Proposition~\ref{prop:fluc1}, it will suffice to show~that, if $\gamma \in (0,2/3)$, then, for some appropriate choice of $\varepsilon$, 
\begin{equation}\label{eq:errorto0}
\lim_{t \to \infty} \textrm{E}_{t,\alpha}(s) = 0.
\end{equation}
and
\begin{equation}\label{eq:liminfpos}
\liminf_{t \to \infty} \P\left(s^{\frac{1}{3}}U_{t,\alpha}^\gamma(s) + (1-s)^{\frac{1}{3}}V_{t,\alpha}^\gamma(s) \leq -(1+\varepsilon)z_{\varepsilon,s}\right) > \mathrm{e}^{-\frac{1}{12}(1-\varepsilon)z_{\varepsilon,s}^3}.
\end{equation} 
But, since $\lim_{t \to \infty} \alpha(t)=\infty$ and $\gamma < 1$, \eqref{eq:errorto0} is an immediate consequence of Lemma~\ref{lemma:fluc1a}.
On the other hand, to show \eqref{eq:liminfpos}, we first note that $U^{\gamma}_{t,\alpha}(s)$ and $V^{\gamma}_{t,\alpha}(s)$ are independent, being measurable functions of the Poisson processes $(Y^{(r)})_{r \in \N}$ on the non-overlapping time intervals $[0,st]$ and $[st,t]$, respectively. Hence, by Lemma~\ref{lemma:fluc1a}, $U^{\gamma}_{t,\alpha}(s)$ and $V^{\gamma}_{t,\alpha}(s)$ satisfy the hypotheses of Lemma~\ref{lemma:fluc1b}. In particular, if we choose $\tilde{x}:=\max\{ x_\varepsilon^*,(1+\varepsilon)x_\varepsilon\}$, with $x_\varepsilon^*$ as in Lemma~\ref{lemma:fluc1b} and $x_\varepsilon$ as in Theorem~\ref{theo:3}, and in addition set $x'_\varepsilon:=\frac{\tilde{x}}{1+\varepsilon}$, then $(1+\varepsilon)z_{\varepsilon,s}=z_{\varepsilon,s}(\tilde{x})$ (in the notation of Lemma~\ref{lemma:fluc1b}) so that \eqref{eq:liminfpos} now follows from \eqref{eq:cotaint3} by choosing $\varepsilon$ small enough.

\subsection{Proof of Proposition~\ref{prop:fluc2}}\label{sec:propubf} 
We must show that,  if $\alpha$ verifies $\lim_{t \to \infty} \frac{\alpha(t)}{(\log t)^\rho} = \infty$ for all $\rho > 0$ and $\alpha(t)=o(t^\eta)$ for some $\eta \in (0,\frac{9}{31})$, then, for any $\gamma > 2/3$,  
    \begin{equation*}
        \lim_{t \rightarrow \infty} \P (A_G^\gamma (t, \alpha(t)))  = 1.
    \end{equation*}
Observe that, since $A_G^\gamma(t,k) \subseteq A_G^{\gamma^\prime}(t,k)$ for any $\gamma' \geq \gamma$, it is enough to prove this for all $\gamma \in (2/3, 2/3+\varepsilon)$ for some $\varepsilon > 0$ small, i.e., we can assume $\gamma$ is sufficiently close to~$\frac{2}{3}$. To this end, for $t>0$, $k \in \N$ define the sets 
 \begin{align*}
     C_\uparrow^\gamma(t,k) &:= \left\{ (s,n) \in [0, t] \times \N :  n - \frac{k}{t}s \leq k^\gamma \right\},
         \\  C^{\gamma}_\downarrow(t,k) &:= \left\{ (s,n) \in [0, t] \times \N :  n - \frac{k}{t}s\geq -k^\gamma \right\},  
\end{align*}
together with, for $G \in \cI$, the events
\begin{align*}
    B^{\gamma,\uparrow}_G(t,k) &:= \{\exists \,u \in \Pi_G(t,k) : \text{graph}(u) \notin C_\uparrow^\gamma(t,k)\},
     \\B^{\gamma,\downarrow}_G(t,k) &:= \{ \exists \,u \in \Pi_G(t,k) : \text{graph}(u) \notin C_\downarrow^\gamma(t,k)\}.
 \end{align*}
Notice that $(A_G(t,k))^c=B^{\gamma,\uparrow}_G(t,k) \cup B_G^{\gamma,\downarrow}(t,k)$ and, therefore, it is enough to show that
\begin{equation}\label{eq:limit1}
\lim_{t \to \infty} \P( B^{\gamma,\uparrow}_G(t,\alpha(t)))=0
\end{equation} and
\begin{equation}\label{eq:limit2}
\lim_{t \to \infty} \P( B^{\gamma,\downarrow}_G(t,\alpha(t)))=0.
\end{equation} 
Let us start with the limit in~\eqref{eq:limit1}.  Fix $t > 0$ and define a sequence of points $(z^j)_{j=0,\dots,K}$ by the formula
 \begin{equation}\label{eq:defzj}
 z^j := (z^j_{\text{time}},z^j_{\text{space}}):= \left( j\frac{M}{K}, \left\lfloor \frac{\alpha(t)}{t}\bigg(j\frac{M}{K}\bigg)+ (\alpha(t))^\gamma \right\rfloor \right),
 \end{equation}
where $K:=\lceil 2t^2(\alpha(t))^\gamma\rceil$ and $M:= t -  t (\alpha (t))^{\gamma - 1}$. Note that $z^j \in C^\gamma_\uparrow(t.\alpha(t))$ and that
\begin{equation}\label{eq:defmax}
\left\lfloor \frac{\alpha(t)}{t}\bigg(j\frac{M}{K}\bigg)+ (\alpha(t))^\gamma \right\rfloor=\max \{ n \in \N : (j\tfrac{M}{K},n) \in C^\gamma_\uparrow(t,\alpha(t))\},
\end{equation} i.e., $z^j_{\text{space}}$ is the number of the furthest column that can be reached at time $t':=j\frac{M}{K}$ by a geodesic whose graph lies in $C^\gamma_\uparrow(t,\alpha(t))$ at time $t'$. In addition, note that $M$ satisfies
\[
M:=\max \left\{ t' \in [0,t] : \frac{\alpha(t)}{t}t' + (\alpha(t))^\gamma \leq \alpha(t)\right\},
\] i.e., $M \in [0,t]$ is the largest time at which any path $\beta \in \cU(t,\alpha(t))$ can exit $C^\gamma_{\uparrow}(t,\alpha(t))$.
Then, for $j=1,\dots,K$, let
\[
D_j:=\sum_{r=1}^{\alpha(t)} Y^{(r)}_{j \frac{M}{K}} - Y^{(r)}_{(j-1)\frac{M}{K}}
\] denote the number of marks in the rectangle $((j-1)\frac{M}{K},j\frac{M}{K}] \times [1,\alpha(t)]$ of $Y=(Y^{(r)})_{r \in \N}$ (when viewed as a single Poisson process on $\R_{\geq 0} \times \N$ of intensity $1$).  

Next, suppose that $\beta \in \Pi_G(t, \alpha(t))$ is such that $\text{graph}(\beta) \notin C^\gamma_\uparrow(t,\alpha(t))$ and define
\[
s_\beta:= \inf\{ s \in [0,t]: (s,\beta(s)) \notin C^\gamma_\uparrow(t,\alpha(t))\}.
\] Notice that, by definition of $M$ and since $\beta(s)=1 \leq (\alpha(t))^{\gamma}$ for all $s$ sufficiently small by definition of $\cU(t,\alpha(t))$, we have $s_\beta \in (0,M]$. Hence, we may define
\[
j_\beta:= \max \{ j=0,\dots,K : s_\beta > j\tfrac{M}{K}\}
\] and set $z:=z^{j_\beta}$ (where $z^j$ are defined as in \eqref{eq:defzj}). 
For this choice of $z$, in light of~\eqref{eq:defmax} (and the remark in the line below) and since $|z^{j+1}_{\text{space}}-z^{j}_{\text{space}}| \leq 1$ for all $j=0,\dots,K-1$ if $t$ is sufficiently large, we obtain that
\begin{equation}
\label{geo_ineq_1}
    h^\circ_G(t, \alpha(t)) \leq
    h^\circ_F( z )+  \max_{i =1,2} h^\circ_F(z + (0,i) \rightarrow (t,\alpha(t))) + \max_{1 \leq j \leq K} D_j
\end{equation} for all $t$ sufficiently large, where $h_F^\circ(z^j):=h^\circ_F(z^j_{\text{time}},z^j_{\text{space}})$ and $h_F^\circ( (t_1,k_1) \to (t_2,k_2))$ is defined as in \eqref{eq:defhpp}. See {Figure~\ref{fig:C-1} for an illustration of \eqref{geo_ineq_1}}.  

\begin{figure}
 \begin{center}
 \includegraphics[scale=.8]{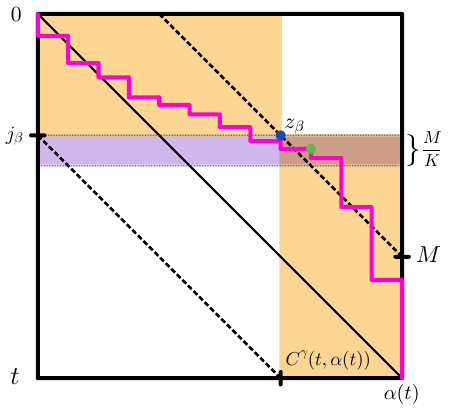}
\caption{\textbf{Illustration of \eqref{geo_ineq_1}}. The picture shows a geodesic $\beta$ for $h^\circ_G(t,\alpha(t))$, colored in pink, escaping the cylinder $C^\gamma(t,\alpha(t))$, the region contained by the two dotted lines in black. The green dot corresponds to the point $z^\circ:=(s_\beta,\beta(s_\beta))$ which marks the exact time and location in which $\beta$ escapes $C^\gamma(t,\alpha(t))$. The point $z$, indicated by a blue dot, acts as ``discrete approximation'' of $z^\circ$. In addition, the picture includes two orange rectangles, $R_1:=[0,z^{j_\beta}_{\text{time}}]\times [0,z^{j_\beta}_{\text{space}}]$ and $R_2:=[z^{j_\beta}_{\text{time}},t] \times [z^{j_\beta}_{\text{space}},\alpha(t)]$, together with a purple strip $S:=[z^{j_\beta}_{\text{time}},z^{j_\beta+1}_{\text{time}}]\times [0,\alpha(t)]$, whose overlap with $R_2$ is colored in a darker tone. The picture shows the number of marks collected by $\beta$ in $R_1$ is at most $h^\circ_F(z)$, those collected in $R_2$ is at most $h^\circ_F( z^\circ \to (t,\alpha(t))) \leq \max_{i=1,2} h(z + (0,i) \to (t,\alpha(t)))$, and those collected in $S$ is at most~$D_{j^\beta}$, which immediately yields \eqref{geo_ineq_1}.}\label{fig:C-1}
\end{center}
\end{figure}

Now, let us set $\delta:=\gamma - \frac{2}{3} \in (0,\varepsilon)$ and, for $j=0,\dots,K$ and $i=1,2$, write
\begin{align*}
\widehat{h^\circ_F}(z^j)&:= h^\circ_F (z^j) - z^j_{\text{time}} - 2\sqrt{ z^j_\text{time} z^j_{\text{space}}},
 \\
 \widetilde{h^\circ_F}(z^j;i) &:= h^\circ_F (z^j +(0,i) \rightarrow (t, \a (t))) - \lp t-z^j_{\text{time}} \rp - 2\sqrt{ \lp t - z^j_\text{time}\rp \lp \a (t) -z^j_{\text{space}} \rp},
\end{align*}
and define the events
\begin{align*}
\Theta^j_t
&:= \left\{ \widehat{h^\circ_F}(z^j) > \sqrt{z^j_\text{time} (z^j_\text{space})^{-\frac{1}{3}}}(z^j_\text{space})^{\delta}+\mathrm{e}^4f(t)\right\},
\\
\Sigma_t^j(i) &:= \ls \widetilde{h^\circ_F}(z^j;i) > \sqrt{ \lp t - z^j_\text{time}\rp \lp \a (t) -z^j_{\text{space}} \rp^{-\frac{1}{3}}} \lp \a (t) -z^j_{\text{space}} \rp^\delta + \mathrm{e}^4 f(t)\rs,
\end{align*}
where $f(t):=\sqrt{t(\alpha(t))^{\delta-\frac{1}{3}}}$, then 
\begin{equation}\label{eq:vanishu1}
\lim_{t \to \infty} \P\left(\bigcup_{0\leq j\leq K} \Theta^j_t \cup \Sigma^j_t(1) \cup \Sigma^j_t(2)\right)=0.
\end{equation} 

Indeed, on the one hand, if $j$ is such that $z^j_{\text{time}} \leq f(t)$,
then, since on $\Theta^j_t$ we have $h^\circ_F(z^j) > \mathrm{e}^4f(t)$ and also $z^j_{\text{space}} \leq 2(\alpha(t))^\gamma \leq f(t)$ for any such $j$ if $t$ is sufficiently large (depending only on $\alpha$), by the inequality $h^\circ_F(z^j) \leq L(z^j_{\text{time}},z^j_{\text{space}})$ with $L$ given by \eqref{eq:defd}, we have that, for all $t$ large,
\begin{equation}
\label{theta_j_leq_f(t)}
\P\left(\bigcup_{j\,:\, z^j_{\text{time}} \leq f(t)} \Theta^j_t\right) \leq \P\left( L\left(f(t), \lfloor f(t)\rfloor\right) > \mathrm{e}^4f(t)\right) \leq \mathrm{e}^{-f(t)}, 
\end{equation}
where the last inequality follows from Lemma~\ref{lemma:3}. 

On the other hand, for all $z^j_{\text{time}} > f(t)$, we may use the moderate deviation estimates from Theorem \ref{theo:3} to bound $\P(\Theta^j_t)$. However, notice that, in order to use Theorem~\ref{theo:3} to control $\widehat{h^\circ_F}(z^j)$, we must make sure that the point $z^j$ lies in the space-time domain in which these estimates hold, i.e., we need to guarantee that $z^j_{\text{space}} \leq \sigma (z^j_{\text{time}})$ for some curve $\sigma$ (which in the statement of Theorem~\ref{theo:3} appears as $\alpha$ but that, in principle, can be different from our current $\alpha$) satisfying $\lim_{t \to \infty} \sigma(t)=\infty$ and $\sigma(t)=o(t^{\frac{3}{7}}(\log t)^{-\frac{6}{7}})$. To this end, we consider the curve $\sigma : \R_{\geq 0} \to \N$ given by 
\[
\sigma(t):= t^{3/7} (\log(t))^{-6/7} \sqrt{r(t)},
\] with $r(t):=\max\{ (\log t)^{-2},\sup_{s \geq t} \frac{\alpha(t)}{t^{\eta}}\}$. Observe that, due to our current assumptions on $\alpha$, we have that $\lim_{t \to \infty} \sigma(t)=\infty$ and $\sigma(t)=o(t^{\frac{3}{7}}(\log t)^{-\frac{6}{7}})$. Furthermore, if $\delta=\gamma-\frac{2}{3}$ is small enough, for all $t$ large enough we have $z^j_{\text{space}} \leq \sigma (z^j_{\text{time}})$ for every $j$ such that  $z^j_{\text{time}} > f(t)$. Indeed, one the one hand, since $z^j_{\text{time}} \leq t$ and $r$ is decreasing by definition,  for all $t$ large enough so that $t^\eta \leq t^{\frac{3}{7}}(\log t)^{-\frac{6}{7}}$,
\[
\frac{\frac{\alpha(t)}{t} z^j_{\text{time}}}{\sigma(z^j_{\text{time}})} \leq \frac{\alpha(t)}{t (\log t)^{-\frac{6}{7}}} \frac{(z^j_{\text{time}})^{\frac{4}{7}}}{\sqrt{r(t)}} \leq \frac{\alpha(t)}{t^{\frac{3}{7}}(\log t)^{-\frac{6}{7}}}\frac{1}{\sqrt{r(t)}} \leq \sqrt{r(t)}.
\] On the other hand, since $z^j_{\text{time}} \geq f(t)$, by choosing $\gamma$ close enough to $\frac{2}{3}$ so that $\eta \leq \frac{3}{14\gamma +1}$, we obtain that, for all $t$ large enough,
\[
\frac{(\alpha(t))^\gamma}{\sigma(z^j_{\text{time}})} \leq \frac{(\alpha(t))^\gamma(\log t)^{\frac{6}{7}}}{(f(t))^{\frac{3}{7}}\sqrt{r(t)}} =
\left(\frac{ \alpha(t)}{t^{\frac{3}{14\gamma+1}}}\right)^{\gamma+\frac{1}{14}} \frac{(\log t)^{\frac{13}{7}}}{(\alpha(t))^{\frac{3}{14}\delta}} \leq (r(t))^{\gamma+\frac{1}{14}},
\] so that, since $r(t) \to 0$ as $t \to \infty$, upon combining the last two displays we conclude that $z^j_{\text{space}} \leq \sigma(z^j_{\text{time}})$ whenever $z^j_{\text{time}} \geq f(t)$ for all $t$ large enough if $\delta$ is small enough. Hence, by Theorem~\ref{theo:3} we obtain that, for any such $j$,
\begin{equation}
\label{theta_j^j}
\P(\Theta^j_t) \leq \P\left( \widehat{h^\circ_F}(z^j) > \sqrt{z^j_\text{time} (z^j_\text{space})^{-\frac{1}{3}}}(z^j_\text{space})^{\delta}\right) \leq \mathrm{e}^{-(\log t)^3}.
\end{equation} if $t$ is large enough since, under our current assumptions on $\alpha$, we have 
\[
(z^j_\text{space})^{\delta} \geq \min \{ C (z^j_{\text{space}})^{\frac{2}{3}},(\log t)^2\} = (\log t)^2
\] for any fixed constant $C > 0$ if $t$ is sufficiently large.

Thus, since $K \leq t^3$ for all $t$ large enough, upon combining \eqref{theta_j_leq_f(t)}--\eqref{theta_j^j} and performing an union bound over all $j=0,\dots,K$, we obtain that 
\[
\lim_{t \to \infty} \P\left( \bigcup_{j=0}^K \Theta^j_t \right)=0.
\]
In a similar fashion one can also check that
\[
\lim_{t \to \infty} \P \lp \bigcup_{j = 0}^K \Sigma_t^j(1) \cup \Sigma^j_t(2) \rp = 0,
\] which, in combination with the previous display, immediately yields \eqref{eq:vanishu1}.

In addition, since each $D_j$ is a Poisson random variable with mean $\alpha(t)\frac{M}{K}=o(1)$, then, by the exponential Tchebychev inequality, 
\begin{equation}\label{eq:vanishu2}
    \P \lp \max_{1 \leq j \leq K} D_j \geq \log(t)\rp \leq K\P( 3D_1 \geq \log (t^3)) \leq \frac{K}{t^3} \E \left[ e^{3D_1}\right] \longrightarrow 0.
\end{equation}
In light of \eqref{eq:vanishu1}-\eqref{eq:vanishu2}, in order to prove \eqref{eq:limit1} it only remains to show that
\begin{equation}
    \lim_{t \to \infty} \P( B_G^{\gamma,\uparrow,\circ}(t,\alpha(t)))=0,
\end{equation} where
\[
B_G^{\gamma,\uparrow,\circ}(t,\alpha(t)):= B_G^{\gamma,\uparrow} \setminus \left( \left\{ \max_{0 \leq j \leq K} D_j \geq \log(t) \right\} \cup \bigcup_{j=0}^K \Theta^j_t \cup \Sigma^j_t(1) \cup \Sigma^j_t(2)\right). 
\]
To this end, observe that by \eqref{geo_ineq_1} we have that, on the event $B_G^{\gamma,\uparrow,\circ}(t,\alpha(t))$, 
\[
    h^\circ_G(t,\alpha(t)) \leq t + 2\sqrt{z_\text{time}z_\text{space}}
    + 2\sqrt{(t-z_\text{time})(\alpha(t) -z_\text{space})} + \Upsilon_t+ 2\mathrm{e}^4f(t)+\log(t),
\]
where
\[
\Upsilon_t:= \sqrt{z_\text{time}(z_{\text{space}})^{-\frac{1}{3}}} (z_{\text{space}})^\delta + \sqrt{(t- z_\text{time})(\alpha(t)-z_\text{space})^{-\frac{1}{3}}} (\alpha(t)-z_\text{space})^\delta.
\]
     
Since the function $s \mapsto s(c_1s+c_2)^{-\frac{1}{3}}$ is increasing on $\R_{>0}$ for any $c_1,c_2 \geq 0$, if $\gamma < 1$, then we have $
\Upsilon_t \leq 4\sqrt{t(\alpha(t))^{-\frac{1}{3}}}(\alpha(t))^{\delta}$ for all $t$ large enough, which, since in addition $\log t \leq 2(\alpha(t))^\gamma \leq f(t) \leq \rme^{-4}\sqrt{t(\alpha(t))^{-\frac{1}{3}}}(\alpha(t))^{\delta}$ holds for all $t$ large enough by definition of $f(t)$ and our assumptions on $\alpha$, shows that on the event $B_G^{\gamma,\uparrow,\circ}(t,\alpha(t))$ one has 
\[
h^\circ_G(t,\alpha(t)) \leq t + 2\sqrt{z_\text{time}z_\text{space}}+ 2\sqrt{(t-z_\text{time})(\alpha(t) -z_\text{space})} + 7\sqrt{t(\alpha(t))^{-\frac{1}{3}}}(\alpha(t))^{\delta}.
\]
Since $z \in [0,M]$ by definition, it then follows from Lemma \ref{lemma_geodesics_2-0} that, on this same event, 
\begin{equation}
    h^\circ_G(t,\alpha(t))
    \leq t + 2\sqrt{t \alpha(t)} - \frac{1}{64}\sqrt{t \a (t)^{-1/3}} (\alpha(t))^{2 \delta} 
\end{equation} for all $t$ large enough. But, since $\lim_{t \to \infty} (\alpha(t))^{\delta}=\infty$, by Theorem~\ref{theo:1} we have that
\[
\lim_{t \to \infty} \P\left( h^\circ_G(t,\alpha(t))
    \leq t + 2\sqrt{t \alpha(t)} - \frac{1}{64}\sqrt{t \a (t)^{-1/3}} (\alpha(t))^{2 \delta} \right)=0
\] which implies that $\lim_{t \to \infty} \P(B_G^{\gamma,\uparrow,\circ}(t,\alpha(t)))=0$ and thus concludes the proof of \eqref{eq:limit1}. 

The argument to prove \eqref{eq:limit2} is almost symmetric. Let us define the sequence of points $(\tilde{z}^j)_{j=0,\dots,K}$ as $\tilde{z}^j:=(t,\alpha(t))-z^j$, for $z^j$ as in \eqref{eq:defzj}. Then, suppose that $\beta \in \Pi_G(t,\alpha(t))$ is such that $\text{graph}(\beta) \in C^\gamma_{\downarrow}(t,\alpha(t))$ and define
\[
\tilde{s}_\beta:= \sup \{ s \in [0,t] : (s,\beta(s)) \notin C^\gamma_\downarrow(t,\alpha(t))\} \in [t-M,t],
\] together with $\tilde{j}_\beta:=\min \{ j=0,\dots,K : \tilde{s}_\beta \leq t- j\frac{M}{K}\}$ and $\tilde{z}:=\tilde{z}_{\tilde{j}_\beta}$. Then, by an argument similar to the one yielding \eqref{geo_ineq_1}, one can show that, for all $t$ large enough,
\[
h^\circ_G(t,\alpha(t)) \leq h^\circ_F(\tilde{z}+(0,1)) + \max_{i=0,1} h^\circ_F(\tilde{z}-(\tfrac{M}{K},i) \to (t,\alpha(t))) + \max_{1 \leq j \leq K} D_j
\] From here, the proof follows follows the same lines as that of \eqref{eq:limit1}, we omit the details.

\newpage

\centerline{\bf Acknowledgements}
Part of this work was carried out during visits of some of the authors to NYU Shanghai, Pontificia Universidad Cat\'olica de Chile and Universidad de Buenos Aires. The authors would like to thank the institutions for their hospitality and financial support.
Alejandro Ram\'irez was partially supported by NFSC 12471147 grant and by NYU Shanghai Boost Fund.
Pablo Groisman and Sebasti\'an Zaninovich were partially supported by CONICET Grant PIP 2021 11220200102825CO, UBACyT Grant 20020190100293BA and PICT 2021-00113 from Agencia I+D. Santiago Saglietti was partially supported by Fondecyt Grant 1240848.
\bibliographystyle{abbrv}
\bibliography{biblio}
\end{document}